\documentclass[a4paper]{amsart}

\usepackage{epic, eepic, amsfonts, latexsym, amssymb, graphicx,
multicol, mathrsfs, color, amscd, verbatim, paralist,
xspace, url, euscript, stmaryrd,  amsmath, enumitem,
bbold, multirow, tikz, mathtools}
\usepackage[all,pdf,cmtip]{xy}

\usepackage[colorlinks, linkcolor=blue, citecolor=magenta, urlcolor=cyan]{hyperref}

\usepackage{chngcntr}
\usepackage{apptools}

\def\tilde{\widetilde}
\def\hat{\widehat}
\renewcommand\bar{\overline}

\def\RR{{\mathbb R}}
\def\CC{{\mathbb C}}

\def\hat{\widehat}

\def\GL{\mathop{\rm GL}\nolimits}

\def\O{\mathop{\rm O}\nolimits}
\def\SO{\mathop{\rm SO}\nolimits}
\def\U{\mathrm U}

\def\Aut{\mathop{\rm Aut}\nolimits}
\def\Im{\mathop{\rm Im}\nolimits}
\def\Re{\mathop{\rm Re}\nolimits}
\def\ad{\mathop{\rm ad}\nolimits}

\def\Aff{\mathop{\rm Aff}\nolimits}

\def\tr{\mathop{\rm tr}\nolimits}

\def\qed{{\hfill $\Box$}}
\newtheorem{theorem}{THEOREM}[section]

\theoremstyle{definition}

\newtheorem{lemma}[theorem]{Lemma}

\theoremstyle{remark}
\newtheorem{remark}[theorem]{Remark}

\makeatletter
\def\blfootnote{\xdef\@thefnmark{}\@footnotetext}
\makeatother

\begin{document}

\title[Manifolds with high-dimensional automorphism group]{Further steps towards classifying\\ homogeneous Kobayashi-hyperbolic manifolds\\ with high-dimensional automorphism group}

\author[Isaev]{Alexander Isaev}

\address{Mathematical Sciences Institute\\
Australian National University\\
Canberra, Acton, ACT 2601, Australia}
\email{alexander.isaev@anu.edu.au}

\maketitle

\thispagestyle{empty}

\pagestyle{myheadings}

\begin{abstract}We determine all connected homogeneous Kobayashi-hyperbolic manifolds of dimension $n\ge 4$ whose group of holomorphic automorphisms has dimension either $n^2-4$, or $n^2-5$, or $n^2-6$. This paper continues a series of articles that achieve classifications for automorphism group dimension $n^2-3$ and greater.
\end{abstract}

\section{Introduction}\label{intro}
\setcounter{equation}{0}

Kobayashi-hyperbolic manifolds (hereafter called just hyperbolic) are of general interest in complex analysis and geometry as they possess many nice properties (see \cite{Ko1}, \cite{Ko3} for details). For instance, if $M$ is hyperbolic, the group $\Aut(M)$ of holomorphic automorphisms of $M$ is a Lie group in the compact-open topology. This is a consequence of the fact that the action of $\Aut(M)$ on $M$ is proper, which yields that $\Aut(M)$ is locally compact, and therefore a Lie transformation group (see, e.g., the survey paper \cite{Isa5} for details). 

For a hyperbolic manifold $M$, we denote by $n$ its complex dimension and set $d(M):=\dim\Aut(M)$. It is a classical fact that $d(M)$ does not exceed $n^2+2n$, with $d(M)=n^2+2n$ if and only if $M$ is biholomorphic to the unit ball $B^n$ in complex space $\CC^n$ (see \cite[Chapter V, Theorem 2.6]{Ko1}). In papers \cite{Isa1}, \cite{Isa2}, \cite{Isa4}, \cite{IK} we found all hyperbolic manifolds with $n^2-1\le d(M)< n^2+2n$ when $n\ge 2$. Our classification has proved to be useful in applications (see, e.g., \cite{V}), so it would be desirable to extend it to lower automorphism group dimensions. Notice, however, that a generic Reinhardt domain in $\CC^2$ has a 2-dimensional automorphism group, thus no reasonable explicit classification can exist already for $d(M)=n^2-2$, at least when $n=2$. One can try excluding the problematic case $d(M)=2$, $n=2$ from consideration and focus on manifolds of dimension $n\ge 3$, but 
disregarding difficult situations like this one randomly seems to be somewhat artificial. The more natural direction in which we hope some further progress can be made is to introduce the assumption of {\it homogeneity}, i.e., to suppose that the action of $\Aut(M)$ on $M$ is transitive. Homogeneous manifolds are widely considered in geometry, so this assumption, while being restrictive, is a standard one. Clearly, in the homogeneous case one must have $d(M)\ge 2n$.

In \cite[Theorem 1.1]{Isa6} and \cite[Theorem 1.1]{Isa7} we took two steps down from the lowest previously explored automorphism group dimension $n^2-1$ and found all homogeneous hyperbolic manifolds with $d(M)=n^2-2$ and $d(M)=n^2-3$, respectively (where one has to have $n\ge 3$). As one would expect, the lower the value of $d(M)$, the harder it is to produce an explicit classification but we feel that one should be able to progress even further. In this paper we look at dimensions $n^2-4$, $n^2-5$, $n^2-6$ and prove the following three theorems:

\begin{theorem}\label{main}
Let $M$ be a homogeneous hyperbolic manifold satisfying the condition $d(M)=n^2-4$. Then one of the following holds:

\noindent {\rm (i)} $n=4$ and $M$ is biholomorphic to
$B^1\times B^1\times B^1\times B^1$,

\noindent {\rm (ii)} $n=5$ and $M$ is biholomorphic either to $B^1\times B^1\times B^3$ or to
\begin{equation}
\begin{array}{l}
T_5:=\left\{(z_1,z_2,z_3,z_4,z_5)\in\CC^5:(\Im z_1)^2-(\Im z_2)^2-(\Im z_3)^2-\right.\\
\vspace{-0.3cm}\\
\hspace{5cm}\left.(\Im z_4)^2-(\Im z_5)^2>0,\,\,\Im z_1>0\right\},
\end{array}\label{domaint5}
\end{equation}
where the latter is the symmetric bounded domain of type {\rm (}$\hbox{{\rm IV}}_5${\rm )} {\rm (}written in tube form{\rm )},

\noindent {\rm (iii)} $n=6$ and $M$ is biholomorphic to $B^2\times B^4$.
\end{theorem}

\begin{theorem}\label{main1}
There does not exist a homogeneous hyperbolic manifold $M$ with $d(M)=n^2-5$.
\end{theorem} 

\begin{theorem}\label{main2}
Let $M$ be a homogeneous hyperbolic manifold satisfying the condition $d(M)=n^2-6$. Then one of the following holds:

\noindent {\rm (i)} $n=4$ and $M$ is biholomorphic to the domain
\begin{equation}
\begin{array}{l}
{\mathcal D}:=\left\{(z,w)\in\times\CC^3\times\CC:(\Im z_1-|w|^2)^2-(\Im z_2-|w|^2)^2-\right.\\
\vspace{-0.3cm}\\
\hspace{6cm}\left.(\Im z_3)^2>0,\,\,\Im z_1-|w|^2>0\right\},\label{psmathcald}
\end{array}
\end{equation}

\noindent {\rm (ii)} $n=5$ and $M$ is biholomorphic to $B^1\times B^2 \times B^2$,

\noindent {\rm (iii)} $n=6$ and $M$ is biholomorphic to either $B^3\times B^3$, or $B^1\times B^1\times B^4$,

\noindent {\rm (iv)} $n=7$ and $M$ is biholomorphic to $B^2\times B^5$.
\end{theorem}

Combining these theorems with the classifications found earlier, namely the classical result for $d(M)=n^2+2n$ mentioned above, \cite[Theorem 2.2]{Isa3}, \cite[Theorem 1.1]{Isa6} and \cite[Theorem 1.1]{Isa7}, we obtain:

\begin{theorem}\label{combined}
Let $M$ be a homogeneous hyperbolic manifold satisfying\linebreak $n^2-6\le d(M)\le n^2+2n$. Then $M$ is biholomorphic either to the domain ${\mathcal D}$ introduced in {\rm (\ref{psmathcald})} {\rm (}here $n=4$, $d(M)=10=n^2-6${\rm )}, or to a suitable product of unit balls, or to a suitable symmetric bounded domain of type {\rm (IV)}, or to a suitable product of a unit ball and a symmetric bounded domain of type {\rm (IV)}. Specifically, the following products of unit balls are possible:
\begin{itemize}

\item[{\rm (i)}] $B^n$ {\rm (}here $d(M)=n^2+2n${\rm )},
\vspace{0.1cm}

\item[{\rm (ii)}] $B^1\times B^{n-1}$ {\rm (}here $d(M)=n^2+2${\rm )},
\vspace{0.1cm}

\item[{\rm (iii)}] $B^1\times B^1\times B^1$ {\rm (}here $n=3$, $d(M)=9=n^2${\rm )},
\vspace{0.1cm}

\item[{\rm (iv)}] $B^2\times B^2$ {\rm (}here $n=4$, $d(M)=16=n^2${\rm )},
\vspace{0.1cm}

\item[{\rm (v)}] $B^1\times B^1\times B^2$ {\rm (}here $n=4$, $d(M)=14=n^2-2${\rm )},
\vspace{0.1cm}

\item[{\rm (vi)}] $B^2\times B^3$ {\rm (}here $n=5$, $d(M)=23=n^2-2${\rm )}, 

\item[{\rm (vii)}] $B^1\times B^1\times B^1\times B^1$ {\rm (}here $n=4$, $d(M)=12=n^2-4${\rm )},

\item[{\rm (viii)}] $B^1\times B^1\times B^3$ {\rm (}here $n=5$, $d(M)=21=n^2-4${\rm )},

\item[{\rm (ix)}] $B^2\times B^4$ {\rm (}here $n=6$, $d(M)=32=n^2-4${\rm )}, 

\item[{\rm (x)}] $B^1\times B^2\times B^2$ {\rm (}here $n=5$, $d(M)=19=n^2-6${\rm )}, 

\item[{\rm (xi)}] $B^3\times B^3$ {\rm (}here $n=6$, $d(M)=30=n^2-6${\rm )},

\item[{\rm (xii)}] $B^1\times B^1\times B^4$ {\rm (}here $n=6$, $d(M)=30=n^2-6${\rm )},

\item[{\rm (xiii)}] $B^2\times B^5$ {\rm (}here $n=7$, $d(M)=43=n^2-6${\rm )},
   \end{itemize}

\noindent the following symmetric bounded domains of type {\rm (IV)} {\rm (}written in tube form{\rm )} are possible:

\begin{itemize}

\item[{\rm (xiv)}] the domain of type {\rm (}$\hbox{{\rm IV}}_3${\rm )}
\begin{equation}
T_3:=\left\{(z_1,z_2,z_3)\in\CC^3:(\Im z_1)^2-(\Im z_2)^2-(\Im z_3)^2>0,\,\,\Im z_1>0\right\}\label{domaint3}
\end{equation}
{\rm (}here $n=3$, $d(M)=10=n^2+1${\rm )},
\vspace{0.3cm}

\item[{\rm (xv)}] the domain of type {\rm (}$\hbox{{\rm IV}}_4${\rm )}
\begin{equation}
\begin{array}{l}
T_4:=\left\{(z_1,z_2,z_3,z_4)\in\CC^4:(\Im z_1)^2-(\Im z_2)^2-\right.\\
\vspace{-0.3cm}\\
\hspace{5.5cm}\left.(\Im z_3)^2-(\Im z_4)^2>0,\,\,\Im z_1>0\right\}
\end{array}\label{domaint4}
\end{equation} 
{\rm (}here $n=4$, $d(M)=15=n^2-1${\rm )},
\vspace{0.3cm}

\item[{\rm (xvi)}] the domain of type {\rm (}$\hbox{{\rm IV}}_5${\rm )}, i.e., the domain $T_5$ defined in {\rm (\ref{domaint5})} {\rm (}here $n=5$, $d(M)=21=n^2-4${\rm )},
\end{itemize}

\noindent and the following product of a unit ball and a symmetric bounded domain of type {\rm (IV)} is possible:

\begin{itemize}

\item[{\rm (xvii)}] $B^1\times T_3$ {\rm (}here $n=4$, $d(M)=13=n^2-3${\rm )}.

\end{itemize}

\end{theorem}

The proofs of Theorems \ref{main}, \ref{main1}, \ref{main2} are contained Sections \ref{proof}, \ref{proof1}, \ref{proof2} respectively, and, just as the proofs of the main theorems of \cite{Isa6}, \cite{Isa7}, rely on reduction to the case of the so-called \emph{Siegel domains of the second kind} introduced by I.~Pyatetskii-Shapiro at the end of the 1950s (see Section \ref{prelim} for details). Indeed, the seminal paper \cite{VGP-S} shows that every homogeneous bounded domain in $\CC^n$ is biholomorphic to an affinely homogeneous Siegel domain of the second kind. Moreover, in \cite{N2} this result was extended to arbitrary homogeneous hyperbolic manifolds, which settled a well-known question asked by S.~Kobayashi (see \cite[p.~127]{Ko1}). Theorems \ref{main}--\ref{main2} are then derived, by a somewhat technical argument, from the description of the Lie algebra of the automorphism group of a Siegel domain of the second kind given in \cite{KMO} and \cite[Chapter V, \S 1--2]{S}.

As shown in Sections \ref{proof}, \ref{proof1}, \ref{proof2}, the proofs of Theorems \ref{main}--\ref{main2} reduce to analyzing certain domains in $\CC^n$ with $n\le 7$. All homogeneous Siegel domains of the second kind of dimension up to 7 were classified in \cite{KT}, so one might hope that our results could be deduced from that classification. However, since \cite{KT} does not contain full details as to how the classification was produced, we chose to give an independent exposition. Also, to the best of our knowledge, the automorphism group dimensions for most of the domains found in \cite{KT} have not been determined. In fact, an essential part of our proofs is to either compute or estimate some of these dimensions.

It is clear from Theorem \ref{combined} that for $d(M)\ge n^2-5$ only symmetric domains occur. In contrast, the list for $d(M)=n^2-6$ contains a non-symmetric entry. Indeed, the domain ${\mathcal D}$ defined in (\ref{psmathcald}) is linearly equivalent to the famous example of a bounded non-symmetric homogeneous domain in $\CC^4$ given by I.~Pyatetskii-Shapiro. As the automorphism group dimension drops even further, the resulting classifications will become more interesting as more non-symmetric domains will appear on the list. It is not clear, however, for how many more steps the classification process will remain tractable. 

{\bf Acknowledgements.} Most of the work on this paper was done during the author's visit to the Steklov Mathematical Institute in Moscow, which we thank for its hospitality. We are also grateful to M.~Jarnicki and P.~Pflug for offering their help with the editorial procedures required to process this paper for publication.

\section{Siegel Domains of the Second Kind}\label{prelim}
\setcounter{equation}{0}

Here we define Siegel domains of the second kind and collect their properties as required for our proofs of Theorems \ref{main}--\ref{main2} in the next three sections. What follows is similar to the exposition given in \cite[Section 2]{Isa6}.

To start with, an open subset $\Omega\subset\RR^k$ is called an \emph{open convex cone} if it is closed with respect to taking linear combinations of its elements with positive coefficients. Such a cone $\Omega$ is called \emph{{\rm (}linearly{\rm )} homogeneous} if the group
$$
G(\Omega):=\{A\in\GL_k(\RR):A\Omega=\Omega\}
$$
of linear automorphisms of $\Omega$ acts transitively on it. Clearly, $G(\Omega)$ is a closed subgroup of $\GL_k(\RR)$, and we denote by ${\mathfrak g}(\Omega)\subset{\mathfrak{gl}}_k(\RR)$ its Lie algebra.

We will be interested in open convex cones not containing entire lines. For such cones the dimension of ${\mathfrak g}(\Omega)$ admits a useful estimate.

\begin{lemma}\label{ourlemma}\it Let $\Omega\subset\RR^k$ be an open convex cone not containing a line. Then
\begin{equation}
\dim {\mathfrak g}(\Omega)\le\displaystyle\frac{k^2}{2}-\frac{k}{2}+1.\label{conineq}
\end{equation}
Furthermore, for $k\ge 2$ the equality in {\rm (\ref{conineq})} is attained if and only if $\Omega$ linearly equivalent to the circular cone
$$
C_k:=\{x\in\RR^k:x_1^2-x_2^2-\cdots-x_k^2>0,\,\,x_1>0\}.
$$
Moreover, if for $k\ge 3$ we set
$$
K:=\frac{(k-2)(k-3)}{2}+k+1,
$$
then the inequality $\dim {\mathfrak g}(\Omega)\ge K$ implies that $\Omega$ is linearly equivalent to $C_k$.
\end{lemma}

\begin{proof} Fix a point ${\mathbf x}\in \Omega$ and consider its isotropy subgroup $G_{{\mathbf x}}(\Omega)\subset G(\Omega)$. This subgroup is compact since it leaves invariant the bounded open set $\Omega\cap({\mathbf x}-\Omega)$. Therefore, changing variables in $\RR^k$ if necessary, we can assume that ${\mathbf x}$ lies in the $x_1$-axis and $G_{{\mathbf x}}(\Omega)$ lies in the orthogonal group $\O_k(\RR)$. The group $\O_k(\RR)$ acts transitively on the unit sphere in $\RR^k$, and the isotropy subgroup of ${\mathbf x}$ under the $\O_k(\RR)$-action is $\O_{k-1}(\RR)$ embedded in $\O_k(\RR)$ with respect to the last $k-1$ variables. Since $G_{{\mathbf x}}(\Omega)\subset \O_{k-1}(\RR)$, we have
$$
\dim G_{{\mathbf x}}(\Omega)\le\dim \O_{k-1}(\RR)=\displaystyle\frac{k^2}{2}-\frac{3k}{2}+1,
$$
which implies inequality (\ref{conineq}).

Next, note that for $k=2$ every open convex cone not containing a line is linearly equivalent to $C_2$, so we assume that $k\ge 3$ and
$$
\dim {\mathfrak g}(\Omega)=\displaystyle\frac{k^2}{2}-\frac{k}{2}+1.
$$
Then 
$$
\dim G_{{\mathbf x}}(\Omega)=\dim \O_{k-1}(\RR)=\displaystyle\frac{k^2}{2}-\frac{3k}{2}+1,
$$
hence $G_{{\mathbf x}}$ contains $\SO_{k-1}(\RR)$. Notice that the $\SO_{k-1}(\RR)$-orbit of every point of $\Omega$ not lying in the $x_1$-axis is a $(k-2)$-sphere contained in a level set $\{x_1=a\}$ for some $a>0$. As $\Omega$ is a union of such spheres, it follows that $\Omega$ is linearly equivalent to $C_k$. 

Assume finally that $k\ge 3$ and $\dim {\mathfrak g}(\Omega)\ge K$. Since the orbit of ${\mathbf x}$ is at most $k$-dimensional, we have $\dim G_{{\mathbf x}}\ge K-k$. If $k\ne 5$, by \cite[Lemma on p.~48]{Ko2} the group $G_{{\mathbf x}}$ contains $\SO_{k-1}(\RR)$, thus, as above, we see that $\Omega$ is linearly equivalent to $C_k$. If $k=5$ then $K-k=4$. In this case, by \cite{Ish}, the group $G_{{\mathbf x}}$ contains either $\SO_4(\RR)$ (which has dimension 6) or a subgroup of $\SO_4(\RR)$ conjugate to $\U(2)$ (which has dimension 4). The orbit of every point of $\Omega$ not lying in the $x_1$-axis, under the action of either subgroup, is a $3$-sphere contained in a level set $\{x_1=a\}$ for some $a>0$, so the proof follows as above.\end{proof}

Next, let
$$
H:\CC^m\times\CC^m\to\CC^k
$$
be a Hermitian form on $\CC^m$ with values in $\CC^k$, where we assume that $H(w,w')$ is linear in $w'$ and anti-linear in $w$. For an open convex cone $\Omega\subset\RR^k$, the form $H$ is called \emph{$\Omega$-Hermitian} if $H(w,w)\in\overline{\Omega}\setminus\{0\}$ for all non-zero $w\in\CC^m$. Observe that if $\Omega$ contains no lines and $H$ is $\Omega$-Hermitian, then there exists a positive-definite linear combination of the components of $H$.

Now, a Siegel domain of the second kind in $\CC^n$ is an unbounded domain of the form
$$
S(\Omega,H):=\left\{(z,w)\in\CC^k\times\CC^{n-k}:\Im z-H(w,w)\in \Omega\right\}
$$
for some $1\le k\le n$, some open convex cone $\Omega\subset\RR^k$ not containing a line, and some $\Omega$-Hermitian form $H$ on $\CC^{n-k}$. For $k=n$ we have $H=0$, so in this case $S(\Omega,H)$ is the tube domain       
$$
\left\{z\in\CC^n:\Im z\in \Omega\right\}.
$$
Such tube domains are often called \emph{Siegel domains of the first kind}. At the other extreme, when $k=1$, the domain $S(\Omega,H)$ is linearly equivalent to
$$
\left\{(z,w)\in\CC\times\CC^{n-1}:\Im z-||w||^2>0\right\},
$$ 
which is an unbounded realization of the unit ball $B^n$ (see \cite[p.~31]{R}). More generally, if $\Omega=\{x\in\RR^k:x_1>0,\dots,x_k>0\}$ and, in addition, $S(\Omega,H)$ is homogeneous, then $S(\Omega,H)$ is linearly equivalent to a product of $k$ unbounded realizations of unit balls as above, hence biholomorphic to a product of unit balls. This result follows from \cite[Theorems A, B, C]{KT} (see \cite{Ka} and and \cite[Theorem 11]{KMO} for details), as well as from \cite{N1}. Note that every Siegel domain of the second kind is linearly equivalent to a domain contained in a product of unbounded realizations of unit balls (see \cite[pp.~23--24]{P-S}), hence is biholomorphic to a bounded domain, and therefore is hyperbolic.

Next, the holomorphic affine automorphisms of Siegel domains of the second kind are described as follows (see \cite[pp.~25-26]{P-S}):

\begin{theorem}\label{Siegelaffautom}
Any holomorphic affine automorphism of $S(\Omega,H)$ has the form
$$
\begin{array}{lll}
z&\mapsto & Az+a+2iH(b,Bw)+iH(b,b),\\
\vspace{-0.3cm}\\
w&\mapsto & Bw+b,
\end{array}
$$
with $a\in\RR^k$, $b\in\CC^{n-k}$, $A\in G(\Omega)$, $B\in\GL_{n-k}(\CC)$, where
\begin{equation}
AH(w,w')=H(Bw,Bw')\label{assoc}
\end{equation}
for all $w,w'\in\CC^{n-k}$.
\end{theorem}

A domain $S(\Omega,H)$ is called \emph{affinely homogeneous} if the group $\Aff(S(\Omega,H))$ of its holomorphic affine automorphisms acts on $S(\Omega,H)$ transitively. Denote by $G(\Omega,H)$ the subgroup of $G(\Omega)$ that consists of all transformations $A\in G(\Omega)$ as in Theorem \ref{Siegelaffautom}, namely, of all elements $A\in G(\Omega)$ for which there exists $B\in\GL_{n-k}(\CC)$ such that (\ref{assoc}) holds. By \cite[Lemma 1.1]{D}, the subgroup $G(\Omega,H)$ is closed in $G(\Omega)$. It is easy to deduce from Theorem \ref{Siegelaffautom} that if $S(\Omega,H)$ is affinely homogeneous, the $G(\Omega,H)$-action is transitive on $\Omega$, so the cone $\Omega$ is homogeneous (see, e.g., \cite[proof of Theorem 8]{KMO}). Conversely, if $G(\Omega,H)$ acts on $\Omega$ transitively, the domain $S(\Omega,H)$ is affinely homogeneous. Clearly, the transitivity of the action of the group $G(\Omega,H)$ on $\Omega$ is equivalent to that of the action of its identity component $G(\Omega,H)^{\circ}$.

As shown in \cite{VGP-S}, \cite{N2}, every homogeneous hyperbolic manifold is biholomorphic to an affinely homogeneous Siegel domain of the second kind. Such a realization is unique up to affine transformations; in general, if two Siegel domains of the second kind are biholomorphic to each other, they are also equivalent by means of a linear transformation of special form (see \cite[Theorem 11]{KMO}). The result of \cite{VGP-S}, \cite{N2} is the basis of our proofs of Theorems \ref{main}--\ref{main2} in the next three sections.

In addition, our proofs rely on a description of the Lie algebra of the group $\Aut(S(\Omega,H))$ of an arbitrary Siegel domain of the second kind $S(\Omega,H)$. This algebra is isomorphic to the (real) Lie algebra of complete holomorphic vector fields on $S(\Omega,H)$, which we denote by ${\mathfrak g}(S(\Omega,H))$ or, when there is no fear of confusion, simply by ${\mathfrak g}$. The latter algebra has been extensively studied. In particular, we have (see \cite[Theorems 4 and 5]{KMO}):

\begin{theorem}\label{kmoalgebradescr}
The algebra ${\mathfrak g}={\mathfrak g}(S(\Omega,H))$ admits a grading
$$
{\mathfrak g}={\mathfrak g}_{-1}\oplus{\mathfrak g}_{-1/2}\oplus{\mathfrak g}_0\oplus{\mathfrak g}_{1/2}\oplus{\mathfrak g}_1,
$$
with ${\mathfrak g}_{\nu}$ being the eigenspace with eigenvalue $\nu$ of $\ad\partial$, where
$\displaystyle\partial:=z\cdot\frac{\partial}{\partial z}+\frac{1}{2}w\cdot\frac{\partial}{\partial w}$. 
Here
$$
\begin{array}{ll}
{\mathfrak g}_{-1}=\displaystyle\left\{a\cdot\frac{\partial}{\partial z}:a\in\RR^k\right\},&\dim {\mathfrak g}_{-1}=k,\\
\vspace{-0.1cm}\\
{\mathfrak g}_{-1/2}=\displaystyle\left\{2i H(b,w)\cdot\frac{\partial}{\partial z}+b\cdot\frac{\partial}{\partial w}:b\in\CC^{n-k}\right\},&\dim {\mathfrak g}_{-1/2}=2(n-k),
\end{array}
$$
and ${\mathfrak g}_0$ consists of all vector fields of the form
\begin{equation}
(Az)\cdot\frac{\partial}{\partial z}+(Bw)\cdot\frac{\partial}{\partial w},\label{g0}
\end{equation}
with $A\in{\mathfrak g}(\Omega)$, $B\in{\mathfrak{gl}}_{n-k}(\CC)$ and
\begin{equation}
AH(w,w')=H(Bw,w')+H(w,Bw')\label{assoc1}
\end{equation}
for all $w,w'\in\CC^{n-k}$. Furthermore, one has
\begin{equation} 
\dim {\mathfrak g}_{1/2}\le 2(n-k),\qquad \dim {\mathfrak g}_1\le k.\label{estimm}
\end{equation}
\end{theorem}

It is clear that the matrices $A$ that appear in (\ref{g0}) form the Lie algebra of $G(\Omega,H)$ (compare conditions (\ref{assoc}) and (\ref{assoc1})) and that ${\mathfrak g}_{-1}\oplus{\mathfrak g}_{-1/2}\oplus{\mathfrak g}_0$ is isomorphic to the Lie algebra of the group $\Aff(S(\Omega,H))$.

Following \cite{S}, for a pair of matrices $A,B$ satisfying (\ref{assoc1}) we say that $B$ is \emph{associated to $A$} (with respect to $H$). Let ${\mathcal L}$ be the (real) subspace of ${\mathfrak{gl}}_{n-k}(\CC)$ of all matrices associated to the zero matrix in ${\mathfrak g}(\Omega)$, i.e., matrices skew-Hermitian with respect to each component of $H$. Set $s:=\dim {\mathcal L}$. Then we have
\begin{equation}
\dim {\mathfrak g}_0=s+\dim G(\Omega,H)\le s+\dim {\mathfrak g}(\Omega).\label{estim1}
\end{equation}
By Theorem \ref{kmoalgebradescr} and the inequality in (\ref{estim1}) one obtains
\begin{equation}
d(S(\Omega,H))\le 2n-k+s+\dim {\mathfrak g}(\Omega)+\dim {\mathfrak g}_{1/2}+ \dim {\mathfrak g}_1,\label{estim 8}
\end{equation}
which, combined with (\ref{estimm}), leads to
\begin{equation}
d(S(\Omega,H))\le 4n-2k+s+\dim {\mathfrak g}(\Omega).\label{estim2}
\end{equation}
Further, since there exists a positive-definite linear combination, say ${\mathbf H}$, of the components of the Hermitian form $H$, the subspace ${\mathcal L}$ lies in the Lie algebra of matrices skew-Hermitian with respect to ${\mathbf H}$, thus
\begin{equation}
s\le (n-k)^2.\label{ests}
\end{equation}
By (\ref{ests}), inequality (\ref{estim2}) yields
\begin{equation}
d(S(\Omega,H))\le  k^2-2(n+1)k+n^2+4n+\dim {\mathfrak g}(\Omega).\label{estim3}
\end{equation}
Combining (\ref{estim3}) with (\ref{conineq}), we deduce the following useful upper bound:
\begin{equation}
d(S(\Omega,H))\le\displaystyle\frac{3k^2}{2}-\left(2n+\frac{5}{2}\right)k+n^2+4n+1.\label{estim4}
\end{equation}

Next, by \cite[Chapter V, Proposition 2.1]{S} the component ${\mathfrak g}_{1/2}$ of the Lie algebra ${\mathfrak g}={\mathfrak g}(S(\Omega,H))$ is described as follows:

\begin{theorem}\label{descrg1/2}
The subspace ${\mathfrak g}_{1/2}$ consists of all vector fields of the form
$$
2iH(\Phi(\bar z),w)\cdot\frac{\partial}{\partial z}+(\Phi(z)+c(w,w))\cdot\frac{\partial}{\partial w},
$$
where $\Phi:\CC^k\to\CC^{n-k}$ is a $\CC$-linear map such that for every ${\mathbf w}\in\CC^{n-k}$ one has
\begin{equation}
\Phi_{{\mathbf w}}:=\left[x\mapsto\Im H({\mathbf w},\Phi(x)),\,\, x\in\RR^k\right]\in{\mathfrak g}(\Omega),\label{Phiw0}
\end{equation}
and $c:\CC^{n-k}\times\CC^{n-k}\to\CC^{n-k}$ is a symmetric $\CC$-bilinear form on $\CC^{n-k}$ with values in $\CC^{n-k}$ satisfying the condition
\begin{equation}
H(w,c(w',w'))=2iH(\Phi(H(w',w)),w')\label{cond1}
\end{equation}
for all $w,w'\in\CC^{n-k}$. 
\end{theorem}

Further, by \cite[Chapter V, Proposition 2.2]{S}, the component ${\mathfrak g}_1$ of ${\mathfrak g}={\mathfrak g}(S(\Omega,H))$ admits the following description:

\begin{theorem}\label{descrg1}
The subspace ${\mathfrak g}_1$ consists of all vector fields of the form
$$
a(z,z)\cdot\frac{\partial}{\partial z}+b(z,w)\cdot\frac{\partial}{\partial w},
$$
where $a:\RR^k\times\RR^k\to\RR^k$ is a symmetric $\RR$-bilinear form on $\RR^k$ with values in $\RR^k$ {\rm (}which we extend to a symmetric $\CC$-bilinear form on $\CC^k$ with values in $\CC^k${\rm )} such that for every ${\mathbf x}\in\RR^k$ one has
\begin{equation}
A_{{\mathbf x}}:=\left[x\mapsto a({\mathbf x},x),\,\,x\in\RR^k\right]\in{\mathfrak g}(\Omega),\label{idents1}
\end{equation}
and $b:\CC^k\times\CC^{n-k}\to\CC^{n-k}$ is a $\CC$-bilinear map such that, if for ${\mathbf x}\in\RR^k$ one sets
$$
B_{{\mathbf x}}:=\left[w\mapsto\frac{1}{2}b({\mathbf x},w),\,\,w\in\CC^{n-k}\right],\label{idents2}
$$
the following conditions are satisfied:
\begin{itemize}

\item[{\rm (i)}] $B_{{\mathbf x}}$ is associated to $A_{{\mathbf x}}$ and $\Im\tr B_{{\mathbf x}}=0$ for all ${\mathbf x}\in\RR^k$,
\vspace{0.1cm}

\item[{\rm (ii)}] for every pair ${\mathbf w},{\mathbf w}'\in\CC^{n-k}$ one has
$$
B_{{\mathbf w},{\mathbf w}'}:=\left[x\mapsto\Im H({\mathbf w'},b(x,{\mathbf w})),\,\,x\in\RR^k\right]\in{\mathfrak g}(\Omega),
$$

\item[{\rm (iii)}] $H(w,b(H(w',w''),w''))=H(b(H(w'',w),w'),w'')$ for all $w,w',w''\in\CC^{n-k}$.

\end{itemize}
\end{theorem}

Finally, let us recall the well-known classification, up to linear equivalence, of homogeneous convex cones in dimensions $k=2,3,4$ not containing lines (see, e.g., \cite[pp.~38--41]{KT}), which will be also required for our proofs of Theorems \ref{main}--\ref{main2}:

\begin{itemize}

\item [$k=2$:] 

\begin{itemize}

\item[]

$\Omega_1:=\left\{(x_1,x_2)\in\RR^2:x_1>0,\,\,x_2>0\right\}$, where the algebra ${\mathfrak g}(\Omega_1)$ consists of all diagonal matrices, hence $\dim {\mathfrak g}(\Omega_1)=2$,
\end{itemize}
\vspace{0.3cm}

\item [$k=3$:] 

\begin{itemize}

\item[(i)] $\Omega_2:=\left\{(x_1,x_2,x_3)\in\RR^3:x_1>0,\,\,x_2>0,\,\,x_3>0\right\}$, where the algebra ${\mathfrak g}(\Omega_2)$ consists of all diagonal matrices, hence $\dim {\mathfrak g}(\Omega_2)=3$,
\vspace{0.1cm}

\item[(ii)] $\Omega_3:=C_3=\left\{(x_1,x_2,x_3)\in\RR^3:x_1^2-x_2^2-x_3^2>0,\,\,x_1>0\right\}$, where one has ${\mathfrak g}(\Omega_3)={\mathfrak c}({\mathfrak{gl}}_3(\RR))\oplus{\mathfrak o}_{1,2}$, hence $\dim {\mathfrak g}(\Omega_3)=4$; here for any Lie algebra ${\mathfrak h}$ we denote by ${\mathfrak c}({\mathfrak h})$ its center,

\end{itemize}
\vspace{0.3cm}

\item [$k=4$:] 

\begin{itemize}

\item[(i)] $\Omega_4:=\left\{(x_1,x_2,x_3,x_4)\in\RR^4:x_1>0,\,\,x_2>0,\,\,x_3>0,\,\,x_4>0\right\}$, where the algebra ${\mathfrak g}(\Omega_4)$ consists of all diagonal matrices, hence we have $\dim {\mathfrak g}(\Omega_4)=4$,
\vspace{0.1cm}

\item[(ii)] $\Omega_5:=\left\{(x_1,x_2,x_3,x_4)\in\RR^4: x_1^2-x_2^2-x_3^2>0,\,\,x_1>0,\,\,x_4>0\right\}$, where the algebra ${\mathfrak g}(\Omega_5)=\left({\mathfrak c}({\mathfrak{gl}}_3(\RR))\oplus{\mathfrak o}_{1,2}\right)\oplus\RR$ consists of block-diagonal matrices with blocks of sizes $3\times 3$ and $1\times 1$ corresponding to the two summands, hence $\dim {\mathfrak g}(\Omega_5)=5$,
\vspace{0.1cm}

\item[(iii)] $\Omega_6:=C_4=\left\{(x_1,x_2,x_3,x_4)\in\RR^4: x_1^2-x_2^2-x_3^2-x_4^2>0,\,\,x_1>0\right\}$, where ${\mathfrak g}(\Omega_6)={\mathfrak c}({\mathfrak{gl}}_4(\RR))\oplus{\mathfrak o}_{1,3}$ and $\dim {\mathfrak g}(\Omega_6)=7$.
\end{itemize}
\end{itemize}

We are now ready to prove Theorems \ref{main}--\ref{main2}. 

\section{Proof of Theorem \ref{main}}\label{proof}
\setcounter{equation}{0}

By \cite{VGP-S}, \cite{N2}, the manifold $M$ is biholomorphic to an affinely homogeneous Siegel domain of the second kind $S(\Omega,H)$. Since $n^2-4\ge 2n$, it follows that $n\ge 4$. Also, as $M$ is not biholomorphic to $B^n$, we have $k\ge 2$.

The following simple lemma rules out a majority of other possibilities.

\begin{lemma}\label{n5k3} \it For $n\ge 6$ one cannot have $k\ge 3$ and for $n=5$ one cannot have $k=4$. 
\end{lemma}

\begin{proof} To establish the lemma, it suffices to see that for $k\ge 3$, $n\ge 6$, as well as for $k=4$, $n=5$, the right-hand side of inequality (\ref{estim4}) is strictly less than $n^2-4$, i.e., that for such $k,n$ the following holds:
\begin{equation}
\frac{3k^2}{2}-\left(2n+\frac{5}{2}\right)k+4n+5<0.\label{basicinequality}
\end{equation}

Let us study the quadratic function
$$
\varphi(t):=\frac{3t^2}{2}-\left(2n+\frac{5}{2}\right)t+4n+5.
$$
Its discriminant is
$$
{\mathcal D}:=4n^2-14n-\frac{95}{4},
$$
which is easily seen to be positive for $n\ge 5$. Then the zeroes of $\varphi$ are
$$
\begin{array}{l}
\displaystyle t_1:=\frac{2n+\frac{5}{2}-\sqrt{{\mathcal D}}}{3},\\
\vspace{-0.1cm}\\
\displaystyle t_2:=\frac{2n+\frac{5}{2}+\sqrt{{\mathcal D}}}{3}.
\end{array}
$$

To establish the lemma for $n\ge 6$, $k\ge 3$, it suffices to show that: (i) $t_2>n$ for $n\ge 6$ and (ii) $t_1<3$ for $n\ge 6$. The inequality $t_2>n$ means that  
$$
n-\frac{5}{2}< \sqrt{{\mathcal D}},
$$
or, equivalently, that  
$$
n^2-3n-10>0,
$$
which is straightforward to verify for $n\ge 6$. Next, the inequality $t_1<3$ means that
$$
2n-\frac{13}{2}< \sqrt{{\mathcal D}},
$$
or, equivalently, that
$$
n>\frac{11}{2},
$$
which clearly holds for $n\ge 6$.

Finally, the pair $k=4$, $n=5$ obviously satisfies inequality (\ref{basicinequality}). \end{proof}

By Lemma \ref{n5k3}, in order to establish the theorem, we need to consider the following five cases: (1) $k=2$, $n\ge 4$, (2) $k=3$, $n=4$, (3) $k=3$, $n=5$, (4) $k=4$, $n=4$, (5) $k=5$, $n=5$. 
\vspace{-0.1cm}\\

{\bf Case (1).} Suppose that $k=2$, $n\ge 4$. Recall from Section \ref{prelim} that in this case $S(\Omega,H)$ is biholomorphic to a product of two unit balls $B^{\ell}\times B^{n-\ell}$, $1\le\ell\le n-1$. We have
\begin{equation}
d(B^{\ell}\times B^{n-\ell})=2\ell^2-2n\ell+n^2+2n.\label{dimautgrballs}
\end{equation}
Rather than trying to directly find all values of $\ell$ for which the above expression is equal to $n^2-4$, we will argue as follows.

We have that $H=(H_1,H_2)$ is a pair of Hermitian forms on $\CC^{n-2}$. After a linear change of $z$-variables, we may assume that $H_1$ is positive-definite. In this situation, by applying a linear change of $w$-variables, $H_1$, $H_2$ can be simultaneously diagonalized as
$$
H_1(w,w)=||w||^2,\,\,\, H_2(w,w)=\sum_{j=1}^{n-2}\lambda_j|w_j|^2.
$$
If all the eigenvalues of $H_2$ are equal, $S(\Omega,H)$ is linearly equivalent either to
\begin{equation} 
D_1:=\left\{(z,w)\in\CC^2\times\CC^{n-2}:\Im z_1-||w||^2>0,\,\,\Im z_2>0\right\},\label{domaindd1}
\end{equation}
or to
\begin{equation} 
D_2:=\left\{(z,w)\in\CC^2\times\CC^{n-2}:\Im z_1-||w||^2>0,\,\,\Im z_2-||w||^2>0\right\}.\label{domaindd2}
\end{equation} 
The domain $D_1$ is biholomorphic to the product $B^1\times B^{n-1}$, hence we have\linebreak $d(D_1)=n^2+2>n^2-4$, which shows that $S(\Omega,H)$ cannot be equivalent to $D_1$. 

Next, we will observe that $D_2$ is not homogeneous (cf.~\cite[Example 1]{N1}). One way to show this is to compute the connected identity component $G(\Omega_1,(||w||^2,||w||^2))^{\circ}$ of the group $G(\Omega_1,(||w||^2,||w||^2))$. It is straightforward to see that
$$
G(\Omega_1,(||w||^2,||w||^2))^{\circ}=\left\{\left(\begin{array}{cc}
a & 0\\
0 & a
\end{array}
\right),\,\,a>0\right\},
$$ 
and it is then clear that the action of $G(\Omega_1,(||w||^2,||w||^2))^{\circ}$ is not transitive on $\Omega_1$. This proves that $S(\Omega,H)$ cannot be equivalent to $D_2$ either. Therefore, $H_2$ has at least one  pair of distinct eigenvalues.

Next, as $\dim{\mathfrak g}(\Omega)=2$, inequality (\ref{estim2}) yields
\begin{equation}
s\ge n^2-4n-2.\label{estim5}
\end{equation} 
On the other hand, by (\ref{ests}), we have
$$
s\le n^2-4n+4.
$$
More precisely, $s$ is calculated as
\begin{equation}
s=n^2-4n+4-2m,\label{estim7}
\end{equation}
where $m\ge 1$ is the number of pairs of distinct eigenvalues of $H_2$. This fact is a consequence of the following lemma, to which we will repeatedly refer throughout the paper (cf.~\cite[Lemma 3.9]{Isa7}):

\begin{lemma}\label{dimsubspaceskewherm} \it Let ${\mathcal H}$ be a Hermitian matrix of size $r\times r$ and ${\mathcal K}$ the vector space of skew-Hermitian matrices of size $r\times r$ that are at the same time skew-Hermitian with respect to ${\mathcal H}$:
$$
{\mathcal K}:=\left\{B\in{\mathfrak{gl}}_r(\CC): B^T+\bar B=0,\,\,B^T{\mathcal H}+{\mathcal H}\bar B=0\right\}.
$$
Then $\dim{\mathcal K}=r^2-2p$, where $p$ is the number of pairs of distinct eigenvalues of ${\mathcal H}$. Hence, if $\dim{\mathcal K}=r^2$, then ${\mathcal H}$ is a scalar matrix.
\end{lemma}

\begin{proof} Let $\mu_1,\dots,\mu_r$ be the eigenvalues of ${\mathcal H}$. By applying a suitable unitary transformation, we may assume that ${\mathcal H}$ is diagonal. Then if
$$
B=\left(B_{ij}\right),\,\,\, B_{ij}=-\overline{B}_{ji},\,\,\, i,j=1,\dots,r,
$$
is a skew-symmetric matrix, the condition of the skew-symmetricity of $B$ with respect to ${\mathcal H}$ is written as
$$
B_{ij}\mu_i=-\overline{B}_{ji}\mu_j,\,\,\, i,j=1,\dots,r,
$$
which leads to $B_{ij}=0$ if $\mu_i\ne\mu_j$. \end{proof}

By (\ref{estim5}), (\ref{estim7}) it follows that $1\le m\le 3$, thus we have either (a) $n=4$ and $\lambda_1\ne\lambda_2$ (here $m=1$, $s=2$), or (b) $n=5$ and, upon permutation of $w$-variables, $\lambda_1\ne\lambda_2=\lambda_3$ (here $m=2$, $s=5$), or (c) $n=5$ and $\lambda_1$, $\lambda_2$, $\lambda_3$ are pairwise distinct (here $m=3$, $s=3$), or (d) $n=6$ and, upon permutation of $w$-variables, $\lambda_1\ne\lambda_2=\lambda_3=\lambda_4$ (here $m=3$, $s=10$).

Now that we have limited $n$ to the range $4,5,6$, it is easy to find all values of $\ell$ for which the right-hand side of (\ref{dimautgrballs}) is equal to $n^2-4$. It turns out that this can only happen for $n=6$, $\ell=2$, which arises from (d). Thus, we see that Case (1) only contributes $B^2\times B^4$ to the classification of homogeneous hyperbolic $n$-dimensional manifolds with automorphism group dimension $n^2-4$ (here $d(B^2\times B^4)=32=n^2-4$). 
\vspace{0.1cm}

{\bf Case (2).} Suppose that $k=3$, $n=4$. Here $S(\Omega,H)$ is linearly equivalent either to
$$
D_3:=\left\{(z,w)\in\times\CC^3\times\CC:\Im z-v|w|^2\in\Omega_2\right\},\label{domaind7}
$$
where $v=(v_1,v_2,v_3)$ is a non-zero vector in $\RR^3$ with non-negative entries, or to
\begin{equation}
D_4:=\left\{(z,w)\in\times\CC^3\times\CC:\Im z-v|w|^2\in\Omega_3\right\},\label{domaind8}
\end{equation}
where $v=(v_1,v_2,v_3)$ is a vector in $\RR^3$ satisfying $v_1^2\ge v_2^2+v_3^2$, $v_1>0$. Notice, however, that if $S(\Omega,H)$ is equivalent to $D_3$, it must be biholomorphic to the product $B^1\times B^1\times B^2$, which is impossible since $d(B^1\times B^1\times B^2)=14>12=n^2-4$. Thus, $S(\Omega,H)$ is in fact equivalent to the domain $D_4$. 

Suppose first that $v_1^2>v_2^2+v_3^2$, i.e., that $v\in\Omega_3$. As the vector $v$ is an eigenvector of every element of $G(\Omega_3,v|w|^2)$, it then follows that $G(\Omega_3,v|w|^2)$ does not act transitively on $\Omega_3$. Therefore, we have $v_1=\sqrt{v_2^2+v_3^2}\ne 0$, i.e., $v\in\partial\Omega_3\setminus\{0\}$. In this case, by \cite[Lemma 3.8 and Remark 3.9]{Isa6}, as well as by \cite[Lemma 3.6]{Isa7}, we see $d(D_4)=10<12=n^2-4$. Hence, $S(\Omega,H)$ cannot in fact be equivalent to $D_4$, so Case (2) contributes nothing to our classification.

\begin{remark}\label{psexample}
We note that for $v\in\partial\Omega_3\setminus\{0\}$ the domain $D_4$ is linearly equivalent to the domain ${\mathcal D}$ defined in (\ref{psmathcald}), which is linearly equivalent to the famous example of a bounded non-symmetric homogeneous domain in $\CC^4$ given by I.~Pyatetskii-Shapiro in 1959 (see \cite[pp.~26--28]{P-S}).
\end{remark}
\vspace{0.1cm}

{\bf Case (3).} Suppose that $k=3$, $n=5$. Here $S(\Omega,H)$ is linearly equivalent either to
\begin{equation}
D_5:=\left\{(z,w)\in\times\CC^3\times\CC^2:\Im z-{\mathcal H}(w,w)\in\Omega_2\right\},\label{domaind9}
\end{equation}
where ${\mathcal H}$ is an $\Omega_2$-Hermitian form, or to
\begin{equation}
D_6:=\left\{(z,w)\in\times\CC^3\times\CC^2:\Im z-{\mathcal H}(w,w)\in\Omega_3\right\},\label{domaind10}
\end{equation}
where ${\mathcal H}$ is an $\Omega_3$-Hermitian form. If $S(\Omega,H)$ is equivalent to $D_5$, it must be biholomorphic to a product of three unit balls, and it is immediate to see that the only possibility is $B^1\times B^1\times B^3$ with $d(B^1\times B^1\times B^3)=21=n^2-4$.

Assume that $S(\Omega,H)$ is equivalent to the domain $D_6$. By Lemma \ref{dimsubspaceskewherm}, we have either $s\le 2$ or $s=4$. If $s\le 2$, then, recalling that $\dim{\mathfrak g}(\Omega_3)=4$, by inequality (\ref{estim2}) we see
$
d(D_6)\le 20<21=n^2-4.
$
Therefore, one in fact has $s=4$. Let ${\mathcal H}=({\mathcal H}_1,{\mathcal H}_2,{\mathcal H}_3)$ and ${\mathbf H}$ be a positive-definite linear combination of ${\mathcal H}_1$, ${\mathcal H}_2$, ${\mathcal H}_3$. By applying a linear change of the $w$-variables, we can diagonalize ${\mathbf H}$ as\linebreak ${\mathbf H}(w,w)=||w||^2$. By Lemma \ref{dimsubspaceskewherm}, each of the $\CC$-valued Hermitian forms ${\mathcal H}_1$, ${\mathcal H}_2$, ${\mathcal H}_3$ is proportional to ${\mathbf H}$. This shows that ${\mathcal H}(w,w)=v||w||^2$, where $v=(v_1,v_2,v_3)$ is a vector in $\RR^3$ satisfying $v_1^2\ge v_2^2+v_3^2$, $v_1>0$. 

As in Case (2), we now observe that $v$ is an eigenvector of every element of $G(\Omega_3,v||w||^2)$. Then, if $v_1^2>v_2^2+v_3^2$, it follows that $G(\Omega_3,v||w||^2)$ does not act transitively on $\Omega_3$. Therefore $v_1=\sqrt{v_2^2+v_3^2}\ne 0$, i.e., $v\in\partial\Omega_3\setminus\{0\}$. As the group $G(\Omega_3)^{\circ}=\RR_{+}\times\SO_{1,2}^{\circ}$ acts transitively on $\partial\Omega_3\setminus\{0\}$, we can suppose that $v=(1,1,0)$, so ${\mathcal H}(w,w)=(||w||^2,||w||^2,0)$.

We will now prove a lemma that works for domains slightly more general than $D_6$. Namely, let $N\ge 1$, and $\tilde{\mathcal H}$ be an $\Omega_3$-Hermitian form on $\CC^N$ defined as
\begin{equation}
\tilde{\mathcal H}(w,w'):=\left(\sum_{j=1}^N\bar w_jw_j',\sum_{j=1}^N\bar w_jw_j',0\right).\label{hermform1}
\end{equation}
Clearly, if $N=2$ and $v=(1,1,0)$ we have $\tilde{\mathcal H}={\mathcal H}$. Set
\begin{equation}
\tilde D_6:=\left\{(z,w)\in\times\CC^3\times\CC^N:\Im z-\tilde{\mathcal H}(w,w)\in\Omega_3\right\}.\label{domaintilded10}
\end{equation}
We will  now obtain (cf.~\cite[Lemma 3.8]{Isa6}):

\begin{lemma}\label{g12d10}\it For ${\mathfrak g}={\mathfrak g}(\tilde D_6)$ one has ${\mathfrak g}_{1/2}=0$.
\end{lemma}

\begin{proof} We will apply Theorem \ref{descrg1/2} to the cone $\Omega_3$ and the $\Omega_3$-Hermitian form $\tilde{\mathcal H}$. Let $\Phi:\CC^3\to\CC^N$ be a $\CC$-linear map given by a matrix $(\varphi^j_i)$, with $j=1,\dots,N$, $i=1,2,3$. Fixing ${\mathbf w}\in\CC^N$, for $x\in\RR^3$ we compute
$$
\begin{array}{l}
\hspace{-0.1cm}\displaystyle{\mathcal H}({\mathbf w},\Phi(x))=\left(\sum_{j=1}^N\bar {\mathbf w}_j(\varphi^j_1x_1+\varphi^j_2x_2+\varphi^j_3x_3),\sum_{j=1}^N\bar {\mathbf w}_j(\varphi^j_1x_1+\varphi^j_2x_2+\varphi^j_3x_3),0\right)=\\
\vspace{-0.1cm}\\
\displaystyle\left(x_1\cdot\sum_{j=1}^N\bar {\mathbf w}_j\varphi^j_1+x_2\cdot\sum_{j=1}^N\bar {\mathbf w}_j\varphi^j_2+x_3\cdot\sum_{j=1}^N\bar {\mathbf w}_j\varphi^j_3,\right.\\
\vspace{-0.3cm}\\
\hspace{3cm}\displaystyle\left.x_1\cdot\sum_{j=1}^N\bar {\mathbf w}_j\varphi^j_1+x_2\cdot\sum_{j=1}^N\bar {\mathbf w}_j\varphi^j_2+x_3\cdot\sum_{j=1}^N\bar {\mathbf w}_j\varphi^j_3,0\right).
\end{array}
$$
Then from formula (\ref{Phiw0}) we see
$$
\begin{array}{l}
\displaystyle\Phi_{{\mathbf w}}(x)=\left(x_1\cdot\sum_{j=1}^N\Im(\bar {\mathbf w}_j\varphi^j_1)+x_2\cdot\sum_{j=1}^N\Im(\bar {\mathbf w}_j\varphi^j_2)+x_3\cdot\sum_{j=1}^N\Im(\bar {\mathbf w}_j\varphi^j_3),\right.\\
\vspace{-0.3cm}\\
\hspace{2cm}\displaystyle\left.x_1\cdot\sum_{j=1}^N\Im(\bar {\mathbf w}_j\varphi^j_1)+x_2\cdot\sum_{j=1}^N\Im(\bar {\mathbf w}_j\varphi^j_2)+x_3\cdot\sum_{j=1}^N\Im(\bar {\mathbf w}_j\varphi^j_3),0\right).
\end{array}
$$

Recall now that
\begin{equation}
{\mathfrak g}(\Omega_3)={\mathfrak c}({\mathfrak{gl}}_3(\RR))\oplus{\mathfrak o}_{1,2}=\left\{
\left(\begin{array}{lrl}
\lambda & p & q\\
p & \lambda & r\\
q & -r & \lambda
\end{array}
\right),\,\,\,\lambda,p,q,r\in\RR\right\}.\label{alggomega3}
\end{equation}
It is then clear that the condition that $\Phi_{{\mathbf w}}$ lies in ${\mathfrak g}(\Omega_3)$ for every ${\mathbf w}\in\CC^2$ leads to the relations
$$
\sum_{j=1}^N\Im(\bar {\mathbf w}_j\varphi^j_i)\equiv 0,\,\,i=1,2,3,
$$
which yield $\Phi=0$. By formula (\ref{cond1}) we then see that ${\mathfrak g}_{1/2}=0$ as required. \end{proof}

It follows from estimate (\ref{estim 8}), the second inequality in (\ref{estimm}), and Lemma \ref{g12d10} for $N=2$ that for ${\mathcal H}(w,w)=(||w||^2,||w||^2,0)$ we have
\begin{equation}
d(D_6)\le 18<21=n^2-4\label{essstimm10} 
\end{equation}
(here $s=4$ and $\dim{\mathfrak g}(\Omega_3)=4$). This shows that $S(\Omega,H)$ cannot in fact be equivalent to $D_6$, so Case (3) only contributes $B^1\times B^1\times B^3$ to the classification of homogeneous hyperbolic $n$-dimensional manifolds with automorphism group dimension $n^2-4$.

Although this is not required for our proof of Theorem \ref{main}, we will now find the dimension of the component ${\mathfrak g}_1$ of ${\mathfrak g}={\mathfrak g}(\tilde D_6)$ (cf.~\cite[Proposition A3]{Isa6}). This will allow us not only to improve bound (\ref{essstimm10}) but to compute the value $d(D_6)$ precisely, which will be useful when considering lower automorphism group dimensions. Furthermore, the proof below is independently interesting as it contains explicit computations with the fairly bulky formulas supplied by Theorem \ref{descrg1}, which is rarely seen in the literature.

\begin{lemma}\label{g1d10}\it For ${\mathfrak g}={\mathfrak g}(\tilde D_6)$ one has $\dim{\mathfrak g}_1=1$.
\end{lemma}

\begin{proof} We will utilize Theorem \ref{descrg1} for the cone $\Omega_3$ and the $\Omega_3$-Hermitian form $\tilde {\mathcal H}$ given by (\ref{hermform1}). Consider a symmetric $\RR$-bilinear form on $\RR^3$ with values in $\RR^3$:
$$
\begin{array}{l}
a(x,x)=\left(a_{11}^1x_1^2+a_{22}^1x_2^2+a_{33}^1x_3^2+2a_{12}^1x_1x_2+2a_{13}^1x_1x_3+2a_{23}^1x_2x_3,\right.\\
\vspace{-0.3cm}\\
\hspace{2cm}\left.a_{11}^2x_1^2+a_{22}^2x_2^2+a_{33}^2x_3^2+2a_{12}^2x_1x_2+2a_{13}^2x_1x_3+2a_{23}^2x_2x_3,\right.\\
\vspace{-0.3cm}\\
\hspace{2cm}\left. a_{11}^3x_1^2+a_{22}^3x_2^2+a_{33}^3x_3^2+2a_{12}^3x_1x_2+2a_{13}^3x_1x_3+2a_{23}^3x_2x_3\right),
\end{array}\label{symmforma}
$$
where $a_{ij}^{\ell}\in\RR$. Then for a fixed ${\mathbf x}\in\RR^3$, from (\ref{idents1}) we compute
$$
\begin{array}{l} 
A_{{\mathbf x}}(x)=\left(a_{11}^1{\mathbf x}_1x_1+a_{22}^1{\mathbf x}_2x_2+a_{33}^1{\mathbf x}_3x_3+
a_{12}^1{\mathbf x}_1x_2+a_{12}^1{\mathbf x}_2x_1+a_{13}^1{\mathbf x}_1x_3+\right.\\
\vspace{-0.3cm}\\
\hspace{0.8cm}\left.a_{13}^1{\mathbf x}_3x_1+a_{23}^1{\mathbf x}_2x_3+a_{23}^1{\mathbf x}_3x_2, a_{11}^2{\mathbf x}_1x_1+a_{22}^2{\mathbf x}_2x_2+a_{33}^2{\mathbf x}_3x_3+a_{12}^2{\mathbf x}_1x_2+\right.\\
\vspace{-0.3cm}\\
\hspace{0.8cm}\left.a_{12}^2{\mathbf x}_2x_1+a_{13}^2{\mathbf x}_1x_3+a_{13}^2{\mathbf x}_3x_1+a_{23}^2{\mathbf x}_2x_3+a_{23}^2{\mathbf x}_3x_2,
a_{11}^3{\mathbf x}_1x_1+a_{22}^3{\mathbf x}_2x_2+\right.\\
\vspace{-0.3cm}\\
\hspace{0.8cm}\left.a_{33}^3{\mathbf x}_3x_3+
a_{12}^3{\mathbf x}_1x_2+a_{12}^3{\mathbf x}_2x_1+a_{13}^3{\mathbf x}_1x_3+a_{13}^3{\mathbf x}_3x_1+a_{23}^3{\mathbf x}_2x_3+a_{23}^3{\mathbf x}_3x_2\right)=\\
\vspace{-0.3cm}\\
\hspace{0.8cm}\left((a_{11}^1{\mathbf x}_1+a_{12}^1{\mathbf x}_2+a_{13}^1{\mathbf x}_3)x_1+(a_{12}^1{\mathbf x}_1+a_{22}^1{\mathbf x}_2+a_{23}^1{\mathbf x}_3)x_2+(a_{13}^1{\mathbf x}_1+a_{23}^1{\mathbf x}_2+\right.\\
\vspace{-0.3cm}\\
\hspace{0.8cm}\left.a_{33}^1{\mathbf x}_3)x_3,
(a_{11}^2{\mathbf x}_1+a_{12}^2{\mathbf x}_2+a_{13}^2{\mathbf x}_3)x_1+(a_{12}^2{\mathbf x}_1+a_{22}^2{\mathbf x}_2+a_{23}^2{\mathbf x}_3)x_2+(a_{13}^2{\mathbf x}_1+\right.\\
\vspace{-0.3cm}\\
\hspace{0.8cm}\left.a_{23}^2{\mathbf x}_2+a_{33}^2{\mathbf x}_3)x_3,
(a_{11}^3{\mathbf x}_1+a_{12}^3{\mathbf x}_2+a_{13}^3{\mathbf x}_3)x_1+(a_{12}^3{\mathbf x}_1+a_{22}^3{\mathbf x}_2+a_{23}^3{\mathbf x}_3)x_2+\right.\\
\vspace{-0.3cm}\\
\hspace{0.8cm}\left.(a_{13}^3{\mathbf x}_1+a_{23}^3{\mathbf x}_2+a_{33}^3{\mathbf x}_3)x_3\right),
\end{array}
$$
where $x\in\RR^3$. By (\ref{alggomega3}), the condition that this map lies in ${\mathfrak g}(\Omega_3)$ for every ${\mathbf x}\in\RR^3$ is equivalent to the relations
\begin{equation}
\begin{array}{l}
a_{11}^1=a_{12}^2=a_{13}^3,\,\,a_{12}^1=a_{22}^2=a_{23}^3,\,\,a_{13}^1=a_{23}^2=a_{33}^3,\\
\vspace{-0.3cm}\\
a_{13}^2=-a_{12}^3,\,\,a_{23}^2=-a_{22}^3,\,\,a_{33}^2=-a_{23}^3,\,\,a_{13}^1=a_{11}^3,\\
\vspace{-0.3cm}\\
a_{23}^1=a_{12}^3,\,\,a_{33}^1=a_{13}^3,\,\,a_{12}^1=a_{11}^2,\,\,a_{22}^1=a_{12}^2,\,\,a_{23}^1=a_{13}^2.
\end{array}\label{relparama1}
\end{equation}

Next, let $b:\CC^3\times\CC^N\to\CC^N$ be a $\CC$-bilinear map with the $j$th component given by a matrix $(b_{i\ell}^j)$, $j,\ell=1,\dots,N$, $i=1,2,3$. For every fixed pair of vectors ${\mathbf w},{\mathbf w}'\in\CC^N$ we then compute
$$
\begin{array}{l}
\displaystyle\tilde{\mathcal H}({\mathbf w}',b(x,{\mathbf w}))=\left(x_1\cdot\sum_{j,\ell=1}^Nb_{1\ell}^j\bar{\mathbf w}_j'{\mathbf w}_{\ell}+x_2\cdot\sum_{j,\ell=1}^Nb_{2\ell}^j\bar{\mathbf w}_j'{\mathbf w}_{\ell}+x_3\cdot\sum_{j,\ell=1}^Nb_{3\ell}^j\bar{\mathbf w}_j'{\mathbf w}_{\ell},\right.\\
\vspace{-0.4cm}\\
\displaystyle\hspace{3cm}\left.x_1\cdot\sum_{j,\ell=1}^Nb_{1\ell}^j\bar{\mathbf w}_j'{\mathbf w}_{\ell}+x_2\cdot\sum_{j,\ell=1}^Nb_{2\ell}^j\bar{\mathbf w}_j'{\mathbf w}_{\ell}+x_3\cdot\sum_{j,\ell=1}^Nb_{3\ell}^j\bar{\mathbf w}_j'{\mathbf w}_{\ell},0\right).
\end{array}
$$
Then from (ii) of Theorem \ref{descrg1} we obtain
$$
\begin{array}{l}
\hspace{-0.1cm}\displaystyle B_{{\mathbf w},{\mathbf w}'}(x)\hspace{-0.1cm}=\hspace{-0.15cm}\left(x_1\cdot\sum_{j,\ell=1}^N\Im(b_{1\ell}^j\bar{\mathbf w}_j'{\mathbf w}_{\ell})+x_2\cdot\sum_{j,\ell=1}^N\Im(b_{2\ell}^j\bar{\mathbf w}_j'{\mathbf w}_{\ell})+x_3\cdot\sum_{j,\ell=1}^N\Im(b_{3\ell}^j\bar{\mathbf w}_j'{\mathbf w}_{\ell}),\right.\\
\vspace{-0.4cm}\\
\displaystyle\hspace{1cm}\left.x_1\cdot\sum_{j,\ell=1}^N\Im(b_{1\ell}^j\bar{\mathbf w}_j'{\mathbf w}_{\ell})+x_2\cdot\sum_{j,\ell=1}^N\Im(b_{2\ell}^j\bar{\mathbf w}_j'{\mathbf w}_{\ell})+x_3\cdot\sum_{j,\ell=1}^N\Im(b_{3\ell}^j\bar{\mathbf w}_j'{\mathbf w}_{\ell}),0\right).
\end{array}
$$
Now, the condition that this map lies in ${\mathfrak g}(\Omega_3)$ for all ${\mathbf w},{\mathbf w}'\in\CC^N$ is easily seen to be equivalent to $b=0$. Hence $B_{{\mathbf x}}=0$ for every ${\mathbf x}\in\RR^3$.

We will now utilize the requirement that $B_{{\mathbf x}}=0$ is associated to $A_{{\mathbf x}}$ with respect to ${\mathcal H}$ for every ${\mathbf x}\in\RR^3$ (see condition (i) in Theorem \ref{descrg1}). This requirement is immediately seen to be equivalent to the relations
$$
\begin{array}{l}
a_{12}^1=-a_{11}^1,\,\,a_{22}^1=-a_{12}^1,\,\,a_{23}^1=-a_{13}^1,\\
\vspace{-0.3cm}\\
a_{12}^2=-a_{11}^2,\,\,a_{22}^2=-a_{12}^2,\,\,a_{23}^2=-a_{13}^2,\\
\vspace{-0.3cm}\\
a_{12}^3=-a_{11}^3,\,\,a_{22}^3=-a_{12}^3,\,\,a_{23}^3=-a_{13}^3.\\
\end{array}
$$
Together with (\ref{relparama1}), these relations imply that each $a_{ij}^{\ell}$ is either zero or equal to $\pm a_{11}^1$ as follows:
$$
\begin{array}{l}
a_{22}^1=a_{11}^1,\,\,a_{33}^1=a_{11}^1,\,\,a_{12}^1=-a_{11}^1,\,\,a_{13}^1=0,\,\,\\
\vspace{-0.3cm}\\
a_{23}^1=0,\,\,a_{11}^2=-a_{11}^1,\,\,a_{22}^2=-a_{11}^1,\,\,a_{33}^2=a_{11}^1,\,\,\\
\vspace{-0.3cm}\\
a_{12}^2=a_{11}^1,\,\,a_{13}^2=0,\,\,a_{23}^2=0,\,\,a_{11}^3=0,\,\,a_{22}^3=0,\,\,\\
\vspace{-0.3cm}\\
a_{33}^3=0,\,\,a_{12}^3=0,\,\,a_{13}^3=a_{11}^1,\,\,a_{23}^3=-a_{11}^1.
\end{array}
$$
Therefore,
$$
a(x,x)=a_{11}^1((x_1-x_2)^2+x_3^2,-(x_1-x_2)^2+x_3^2,2(x_1-x_2)x_3).
$$
This shows that $\dim{\mathfrak g}_1=1$ as required. \end{proof}

Next, for $w\in\CC^N$ the proof of \cite[Lemma 3.6]{Isa7} yields
\begin{equation}
\dim G(\Omega_3,(||w||^2,||w||^2,0))=3.\label{dim3any}
\end{equation}
Returning to the case $N=2$, we thus see that if ${\mathcal H}(w,w)=(||w||^2,||w||^2,0)$, then for ${\mathfrak g}={\mathfrak g}(D_6)$ one has $\dim{\mathfrak g}_0=7$ (recall that $s=4$). Combining this fact with Lemmas \ref{g12d10} and \ref{g1d10} for $N=2$, we calculate 
$$
d(D_6)=\dim{\mathfrak g}_{-1}+\dim{\mathfrak g}_{-1/2}+\dim{\mathfrak g}_0+\dim{\mathfrak g}_{1/2}+\dim{\mathfrak g}_1=15,
$$
which significantly improves bound (\ref{essstimm10}) and even gives a precise value for $d(D_6)$.
\vspace{0.1cm}

{\bf Case (4).} Suppose that $k=4$, $n=4$. In this case, after a linear change of variables $S(\Omega,H)$ turns into one of the domains
\begin{equation}
\begin{array}{l}
\left\{z\in\CC^4: \Im z\in\Omega_4\right\},\\
\vspace{-0.1cm}\\
\left\{z\in\CC^4: \Im z\in\Omega_5\right\},\\
\vspace{-0.1cm}\\
\left\{z\in\CC^4: \Im z\in\Omega_6\right\}
\end{array}\label{threetubedomains}
\end{equation}
and therefore is biholomorphic either to $(B^1)^4=B^1\times B^1\times B^1\times B^1$, or to $B^1\times T_3$, or to $T_4$, where $T_3$ and $T_4$ are the tube domains defined in (\ref{domaint3}), (\ref{domaint4}). The dimensions of the automorphism groups of these domains are 12, 13, 15, respectively. Noting that $12=n^2-4$, we see that $S(\Omega,H)$ is in fact biholomorphic to the product $(B^1)^4$, so Case (4) only contributes $(B^1)^4$ to the classification of homogeneous hyperbolic $n$-dimensional manifolds with automorphism group dimension $n^2-4$. 
\vspace{0.1cm}

{\bf Case (5).} Suppose that $k=5$, $n=5$. In this situation inequality (\ref{estim2}) implies $\dim {\mathfrak g}(\Omega)\ge 11=k^2/2-k/2+1$, which by Lemma \ref{ourlemma} yields that $\Omega$ is linearly equivalent to the circular cone $C_5$. Hence, after a linear change of variables $S(\Omega,H)$ turns into the domain
$$
\left\{z\in\CC^5: \Im z\in C_5\right\},
$$
which is the tube domain $T_5$ defined in (\ref{domaint5}). Notice that $d(T_5)=21=n^2-4$, so Case (5) contributes $T_5$ to our classification.

The proof of Theorem \ref{main} is now complete.\qed

\section{Proof of Theorem \ref{main1}}\label{proof1}
\setcounter{equation}{0}

As before, we use the fact that, by \cite{VGP-S}, \cite{N2}, the manifold $M$ is biholomorphic to an affinely homogeneous Siegel domain of the second kind $S(\Omega,H)$. Since one has $n^2-5\ge 2n$, it follows that $n\ge 4$. Also, as $M$ is not biholomorphic to $B^n$, we have $k\ge 2$. The following elementary lemma is analogous to Lemma \ref{n5k3}, and we state it without proof.

\begin{lemma}\label{n5k31} \it For $n\ge 7$ one cannot have $k\ge 3$ and for $n=6$ one cannot have $k=4,5,6$. 
\end{lemma}

By Lemma \ref{n5k31}, in order to establish the theorem, we need to consider the following seven cases: (1) $k=2$, $n\ge 4$, (2) $k=3$, $n=4$, (3) $k=3$, $n=5$, (4) $k=3$, $n=6$, (5) $k=4$, $n=4$, (6) $k=4$, $n=5$, (7) $k=5$, $n=5$. 
\vspace{-0.1cm}\\

{\bf Case (1).} Suppose that $k=2$, $n\ge 4$. This situation is treated analogously to Case (1) considered in Section \ref{proof}. Indeed, $S(\Omega,H)$ cannot be equivalent to the domain $D_1$ defined in (\ref{domaindd1}) because $d(D_1)=n^2+2>n^2-5$ and cannot be equivalent to the domain $D_2$ introduced in (\ref{domaindd2}) since $D_2$ is not homogeneous.

Next, inequality (\ref{estim2}) yields $s\ge n^2-4n-3$, hence by (\ref{estim7}) it follows, as before, that $1\le m\le 3$, which leads us to considering the same four subcases (a)--(d) as in Section \ref{proof}. Each of them is easily seen to make no contributions to our\linebreak classification.
\vspace{0.1cm}

{\bf Case (2).} Suppose that $k=3$, $n=4$. We deal with this situation analogously to Case (2) considered in Section \ref{proof} and, as $d(D_4)=10<11=n^2-5$ for $v\in\partial\Omega_3\setminus\{0\}$, immediately see that this case contributes nothing to the classification either.
\vspace{0.1cm}

{\bf Case (3).} Suppose that $k=3$, $n=5$. We will approach this situation analogously to Case (3) considered in Section \ref{proof}. As no product of three unit balls in $\CC^5$ has automorphism group dimension $20=n^2-5$, the domain $S(\Omega,H)$ must be equivalent to the domain $D_6$ defined in (\ref{domaind10}). By Lemma \ref{dimsubspaceskewherm}, we have either $s=1$, or $s=2$, or $s=4$. If $s=1$, then, recalling that $\dim{\mathfrak g}(\Omega_3)=4$, by inequality (\ref{estim2}) we see $d(D_6)\le 19<20=n^2-5$, so in fact we have $s>1$. The case $s=4$ is excluded by arguing as in Section \ref{proof} and observing that in this situation $d(D_6)=15<20=n^2-5$ for $v\in\partial\Omega_3\setminus\{0\}$. Hence we have $s=2$. 

Let ${\mathcal H}=({\mathcal H}_1,{\mathcal H}_2,{\mathcal H}_3)$ and prove

\begin{lemma}\label{simultdiag} \it The forms ${\mathcal H}_1, {\mathcal H}_2, {\mathcal H}_3$ can be simultaneously diagonalized by a linear change of $w$-variables.
\end{lemma}

\begin{proof} Replacing $\Omega_3$ by a linearly equivalent cone, we may assume that ${\mathcal H}_1$ is positive-definite. Further, by applying a linear change of the $w$-variables, we can simultaneously diagonalize ${\mathcal H}_1$, ${\mathcal H}_2$ as
$$
{\mathcal H}_1(w,w)=||w||^2,\,\,\, {\mathcal H}_2(w,w)=\lambda_1|w_1|^2+\lambda_2|w_2|^2.
$$
If $\lambda_1=\lambda_2$, we can also diagonalize ${\mathcal H}_3$ by a unitary transformation, which establishes the lemma in this case.

Assume now that $\lambda_1\ne\lambda_2$. In this situation, every Hermitian matrix $B$ that is also Hermitian with respect to ${\mathcal H}_2$ is diagonal, i.e.,
$$
B=\left(\begin{array}{ll}
ia & 0\\
0 & ib
\end{array}
\right),\,\, a,b,\in\RR.
$$ 
Since $s=2$, every such matrix is also Hermitian with respect to ${\mathcal H}_3$. This immediately implies that ${\mathcal H}_3$ is diagonal. \end{proof}

By Lemma \ref{simultdiag} we have
$$
{\mathcal H}(w,w)=u|w_1|^2+v|w_2|^2,
$$
where $u=(u_1,u_2,u_3)$, $v=(v_1,v_2,v_3)$ are vectors in $\RR^3$ satisfying $u_1^2\ge u_2^2+u_3^2$, $v_1^2\ge v_2^2+v_3^2$, $u_1>0$, $v_1>0$. We will now prove that $\dim G(\Omega_3,{\mathcal H})\le 2$ by studying the Lie algebra of $G(\Omega_3,{\mathcal H})$, which we momentarily denote  by $\mathfrak{h}$. Recall that ${\mathfrak h}$ consists of all matrices $A\in\mathfrak{g}(\Omega_3)$ satisfying (\ref{assoc1}) for some $B\in\mathfrak{gl}_2(\CC)$. 

Suppose first that $u_1^2> u_2^2+u_3^2$, i.e., $u\in\Omega_3$. Since $\Omega_3$ is homogeneous, we may assume that $u=(1,0,0)$, i.e.,
$$
{\mathcal H}(w,w)=(|w_1|^2+v_1|w_2|^2,v_2|w_2|^2,v_3|w_2|^2).
$$ 
Recalling that $\mathfrak{g}(\Omega_3)$ is given by formula (\ref{alggomega3}), we immediately see that for $A\in\mathfrak{h}$ one must have $p=0$, $q=0$. Therefore, $\dim G(\Omega_3,{\mathcal H}) \le 2$. The same conclusion holds if $v_1^2> v_2^2+v_3^2$, i.e., if $v\in\Omega_3$

Next, let $u_1=\sqrt{u_2^2+u_3^2}\ne 0$ and $v_1=\sqrt{v_2^2+v_3^2}\ne 0$, i.e., $u,v\in\partial\Omega_3\setminus\{0\}$. Since the group $G(\Omega_3)^{\circ}=\RR_{+}\times\SO_{1,2}^{\circ}$ acts transitively on $\partial\Omega_3\setminus\{0\}$, we can assume that $u=(1,1,0)$, i.e.,
$$
{\mathcal H}(w,w)=(|w_1|^2+v_1|w_2|^2,|w_1|^2+v_2|w_2|^2,v_3|w_2|^2).
$$
Suppose now that a matrix
$$
A=\left(\begin{array}{lrl}
\lambda & p & q\\
p & \lambda & r\\
q & -r & \lambda
\end{array}
\right)\in{\mathfrak g}(\Omega_3)\label{matrixAomega3}
$$
lies in ${\mathfrak h}$, i.e., that for some $b_{ij}\in\CC$, $i,j=1,2$ the following holds:
$$
\begin{array}{l}
\lambda(|w_1|^2+v_1|w_2|^2)+p(|w_1|^2+v_2|w_2|^2)+qv_3|w_2|^2=\\
\vspace{-0.1cm}\\
\hspace{4cm}2\Re\Bigl((b_{11}w_1+b_{12}w_2)\bar w_1+v_1(b_{21}w_1+b_{22}w_2)\bar w_2\Bigr),\\
\vspace{-0.1cm}\\
p(|w_1|^2+v_1|w_2|^2)+\lambda(|w_1|^2+v_2|w_2|^2)+rv_3|w_2|^2=\\
\vspace{-0.1cm}\\
\hspace{4cm}2\Re\Bigl((b_{11}w_1+b_{12}w_2)\bar w_1+v_2(b_{21}w_1+b_{22}w_2)\bar w_2\Bigr),\\
\vspace{-0.1cm}\\
q(|w_1|^2+v_1|w_2|^2)-r(|w_1|^2+v_2|w_2|^2)+\lambda v_3|w_2|^2=\\
\vspace{-0.1cm}\\
\hspace{4cm}2\Re\Bigl(v_3(b_{21}w_1+b_{22}w_2)\bar w_2\Bigr).
\end{array}
$$
These conditions, in particular, imply
\begin{equation}
\begin{array}{l}
\lambda v_1+pv_2+qv_3=2 v_1\Re b_{22},\,\,p v_1+\lambda v_2+rv_3=2 v_2\Re b_{22},\\
\vspace{-0.3cm}\\
q=r,\,\, q v_1-rv_2+\lambda v_3=2 v_3\Re b_{22}.
\end{array}\label{basicids}
\end{equation}

If $v_3\ne 0$, the first three identities in (\ref{basicids}) yield
$$
q=r,\,\,p=\frac{v_2-v_1}{v_3}r,
$$
hence $\dim G(\Omega_3,{\mathcal H}) \le 2$. If $v_3=0$, then $v_2=\pm v_1$ and, since $s=2$, we in fact have $v_2=-v_1$. The last two identities in (\ref{basicids}) then immediately imply $q=r=0$. Hence, again, $\dim G(\Omega_3,{\mathcal H}) \le 2$ as required.

We now conclude that the action of $\dim G(\Omega_3,{\mathcal H})$ on $\Omega_3$ is not transitive, so Case (3) contributes nothing to the sought-after classification.
\vspace{0.1cm}

{\bf Case (4).} Suppose that $k=3$, $n=6$. By (\ref{estim2}) we have $s+\dim{\mathfrak g}(\Omega)\ge 13$. On the other hand, $s\le 9$ by (\ref{ests}). Since $\dim{\mathfrak g}(\Omega_2)=3$, $\dim{\mathfrak g}(\Omega_3)=4$, it follows that $\Omega$ is linearly equivalent to $\Omega_3$ and $s=9$. In particular, $S(\Omega,H)$ is linearly equivalent to the domain
\begin{equation}
D_7:=\left\{(z,w)\in\CC^3\times\CC^3: \Im z-{\mathcal H}(w,w)\in\Omega_3\right\},\label{domaind7new}
\end{equation}
where ${\mathcal H}$ is an $\Omega_3$-Hermitian form. 

We will now proceed as in Case (3) considered in Section \ref{proof}. Let ${\mathcal H}=({\mathcal H}_1,{\mathcal H}_2,{\mathcal H}_3)$ and ${\mathbf H}$ be a positive-definite linear combination of ${\mathcal H}_1$, ${\mathcal H}_2$, ${\mathcal H}_3$. By applying a linear change of the $w$-variables, we can diagonalize ${\mathbf H}$ as ${\mathbf H}(w,w)=||w||^2$. Since $s=9$, by Lemma \ref{dimsubspaceskewherm} each of the $\CC$-valued Hermitian forms ${\mathcal H}_1$, ${\mathcal H}_2$, ${\mathcal H}_3$ is proportional to ${\mathbf H}$. This shows that ${\mathcal H}(w,w)=v||w||^2$, where $v=(v_1,v_2,v_3)$ is a vector in $\RR^3$ satisfying $v_1^2\ge v_2^2+v_3^2$, $v_1>0$. 

Observe that $v$ is an eigenvector of every element of $G(\Omega_3,v||w||^2)$. Then, if\linebreak $v_1^2>v_2^2+v_3^2$, it follows that the action of $G(\Omega_3,v||w||^2)$ on $\Omega_3$ is not transitive. Therefore, $v_1=\sqrt{v_2^2+v_3^2}\ne 0$, i.e., $v\in\partial\Omega_3\setminus\{0\}$. As the connected group $G(\Omega_3)^{\circ}=\RR_{+}\times\SO_{1,2}^{\circ}$ acts transitively on $\partial\Omega_3\setminus\{0\}$, we can suppose that $v=(1,1,0)$, i.e., ${\mathcal H}(w,w)=(||w||^2,||w||^2,0)$. In this case the domain $D_7$ coincides with the domain $\tilde D_6$ for $N=3$ (see (\ref{domaintilded10})). Thus, by Lemmas \ref{g12d10} and \ref{g1d10} we see that for ${\mathfrak g}={\mathfrak g}(D_7)$ one has ${\mathfrak g}_{1/2}=0$ and $\dim{\mathfrak g}_1=1$. Furthermore, by (\ref{dim3any}) we have $\dim{\mathfrak g}_0=12$ (recall that $s=9$). Combining these facts together, we calculate
$$
d(D_7)=\dim{\mathfrak g}_{-1}+\dim{\mathfrak g}_{-1/2}+\dim{\mathfrak g}_0+\dim{\mathfrak g}_{1/2}+\dim{\mathfrak g}_1=22<31=n^2-5.
$$
This shows that $S(\Omega,H)$ cannot in fact be equivalent to $D_7$, so Case (4) contributes nothing to the classification of homogeneous hyperbolic $n$-dimensional manifolds with automorphism group dimension $n^2-5$.
\vspace{0.1cm}

{\bf Case (5).} Suppose that $k=4$, $n=4$. We deal with this situation analogously to Case (4) considered in Section \ref{proof} and observe that it contributes nothing to our classification. Indeed, for the three tube domains in (\ref{threetubedomains}) the automorphism group dimensions are 12, 13, 15, respectively, and each of these numbers is greater than $11=n^2-5$.
\vspace{0.1cm}

{\bf Case (6).} Suppose that $k=4$, $n=5$. In this situation inequality (\ref{estim2}) implies $\dim {\mathfrak g}(\Omega)\ge 7=k^2/2-k/2+1$, hence it follows, e.g., by Lemma \ref{ourlemma}, that the cone $\Omega$ is linearly equivalent to the circular cone $C_4=\Omega_6$. Therefore, $S(\Omega,H)$ is linearly equivalent to
\begin{equation}
D_8:=\left\{(z,w)\in\times\CC^4\times\CC:\Im z-v|w|^2\in\Omega_6\right\},\label{domaind12}
\end{equation}
where $v=(v_1,v_2,v_3,v_4)$ is a vector in $\RR^4$ satisfying $v_1^2\ge v_2^2+v_3^2+v_4^2$, $v_1>0$. Assume first that $v_1^2>v_2^2+v_3^2+v_4^2$, i.e., that $v\in\Omega_6$. As the vector $v$ is an eigenvector of every element of $G(\Omega_6,v|w|^2)$, it then follows that $G(\Omega_6,v|w|^2)$ does not act transitively on $\Omega_6$. Therefore, we have $v_1=\sqrt{v_2^2+v_3^2+v_4^2}\ne 0$, i.e., $v\in\partial\Omega_6\setminus\{0\}$. Since group $G(\Omega_6)^{\circ}=\RR_{+}\times\SO_{1,3}^{\circ}$ acts transitively on $\partial\Omega_6\setminus\{0\}$, we can suppose that $v=(1,1,0,0)$, i.e., $v|w|^2=(|w|^2,|w|^2,0,0)$. We will now prove an analogue of Lemma \ref{g12d10}.

Set
$$
\hat D_8:=\left\{(z,w)\in\times\CC^4\times\CC^N:\Im z-\hat{\mathcal H}(w,w)\in\Omega_6\right\},\label{domaintilded8}
$$
where $N\ge 1$ and $\hat{\mathcal H}$ is the $\Omega_6$-Hermitian form analogous to the one introduced in (\ref{hermform1}):
$$
\hat{\mathcal H}(w,w'):=\left(\sum_{j=1}^N\bar w_jw_j',\sum_{j=1}^N\bar w_jw_j',0,0\right).\label{hermform1hat}
$$

\begin{lemma}\label{g12d101tilde}\it For ${\mathfrak g}={\mathfrak g}(\hat D_8)$ one has ${\mathfrak g}_{1/2}=0$.
\end{lemma}

\begin{proof} 

We will apply Theorem \ref{descrg1/2} to the cone $\Omega_6$ and the $\Omega_6$-Hermitian form $\hat{\mathcal H}$. Let $\Phi:\CC^4\to\CC^N$ be a $\CC$-linear map given by a matrix $(\varphi^j_i)$, with $j=1,\dots,N$, $i=1,2,3,4$. Fixing ${\mathbf w}\in\CC^N$, for $x\in\RR^4$ we compute
$$
\begin{array}{l}
\hspace{-0.1cm}\displaystyle{\mathcal H}({\mathbf w},\Phi(x))=\left(\sum_{j=1}^N\bar {\mathbf w}_j(\varphi^j_1x_1+\varphi^j_2x_2+\varphi^j_3x_3+\varphi^j_4x_4),\right.\\
\vspace{-0.5cm}\\
\hspace{5.5cm}\displaystyle\left.\sum_{j=1}^N\bar {\mathbf w}_j(\varphi^j_1x_1+\varphi^j_2x_2+\varphi^j_3x_3+\varphi^j_4x_4),0,0\right)=\\
\vspace{-0.3cm}\\
\displaystyle\left(x_1\cdot\sum_{j=1}^N\bar {\mathbf w}_j\varphi^j_1+x_2\cdot\sum_{j=1}^N\bar {\mathbf w}_j\varphi^j_2+x_3\cdot\sum_{j=1}^N\bar {\mathbf w}_j\varphi^j_3+x_4\cdot\sum_{j=1}^N\bar {\mathbf w}_j\varphi^j_4,\right.\\
\vspace{-0.3cm}\\
\hspace{2.5cm}\displaystyle\left.x_1\cdot\sum_{j=1}^N\bar {\mathbf w}_j\varphi^j_1+x_2\cdot\sum_{j=1}^N\bar {\mathbf w}_j\varphi^j_2+x_3\cdot\sum_{j=1}^N\bar {\mathbf w}_j\varphi^j_3+x_4\cdot\sum_{j=1}^N\bar {\mathbf w}_j\varphi^j_4,0,0\right).
\end{array}
$$
Then from formula (\ref{Phiw0}) we see
$$
\begin{array}{l}
\displaystyle\Phi_{{\mathbf w}}(x)=\left(x_1\cdot\sum_{j=1}^N\Im(\bar {\mathbf w}_j\varphi^j_1)+x_2\cdot\sum_{j=1}^N\Im(\bar {\mathbf w}_j\varphi^j_2)+x_3\cdot\sum_{j=1}^N\Im(\bar {\mathbf w}_j\varphi^j_3)+\right.\\
\vspace{-0.3cm}\\
\hspace{1.8cm}\displaystyle\left.x_4\cdot\sum_{j=1}^N\Im(\bar {\mathbf w}_j\varphi^j_4),\,\,x_1\cdot\sum_{j=1}^N\Im(\bar {\mathbf w}_j\varphi^j_1)+x_2\cdot\sum_{j=1}^N\Im(\bar {\mathbf w}_j\varphi^j_2)+\right.\\
\vspace{-0.3cm}\\
\hspace{4cm}\displaystyle\left. x_3\cdot\sum_{j=1}^N\Im(\bar {\mathbf w}_j\varphi^j_3)+x_4\cdot\sum_{j=1}^N\Im(\bar {\mathbf w}_j\varphi^j_4),0,0\right).
\end{array}
$$

Recall now that
\begin{equation}
{\mathfrak g}(\Omega_6)={\mathfrak c}({\mathfrak{gl}}_4(\RR))\oplus{\mathfrak o}_{1,3}=\left\{
\left(\begin{array}{lrrl}
\lambda & p & q & r\\
p & \lambda & s & t\\
q & -s & \lambda & y\\
r & -t & -y & \lambda\\
\end{array}
\right),\,\,\,\lambda,p,q,r,s,t,y\in\RR\right\}.\label{alggomega31}
\end{equation}
It is then clear that the condition that $\Phi_{{\mathbf w}}$ lies in ${\mathfrak g}(\Omega_4)$ for every ${\mathbf w}\in\CC$ leads to the relations
$$
\sum_{j=1}^N\Im(\bar {\mathbf w}_j\varphi_i^j)\equiv 0,\,\,i=1,2,3,4,
$$
which yield $\Phi=0$. By formula (\ref{cond1}) we then see that ${\mathfrak g}_{1/2}=0$ as required. \end{proof}

It follows from estimate (\ref{estim 8}), the second inequality in (\ref{estimm}), and Lemma \ref{g12d101tilde} for $N=1$ that for $v=(1,1,0,0)$ we have
\begin{equation}
d(D_8)\le 18<20=n^2-5\label{d14estim}
\end{equation}
(here $s=1$ and $\dim{\mathfrak g}(\Omega_6)=7$). This shows that Case (6) contributes nothing to our classification.

\begin{remark}\label{sameestim}
Estimate (\ref{d14estim}) can be also obtained by proving that for $v=(1,1,0,0)$ one has $\dim G(\Omega_6,v|w|^2)=5$. The proof is analogous to that of \cite[Lemma 3.6]{Isa7} but uses (\ref{alggomega31}) instead of (\ref{alggomega3}). Therefore, if $v=(1,1,0,0)$, for the algebra $\mathfrak{g}=\mathfrak{g}(D_8)$ we have $\dim\mathfrak{g}_0=6$, which, combined with the second inequality in (\ref{estimm}) and Lemma \ref{g12d101tilde} for $N=1$, improves bound (\ref{d14estim}) to
$$
d(D_8)=\dim{\mathfrak g}_{-1}+\dim{\mathfrak g}_{-1/2}+\dim{\mathfrak g}_0+\dim{\mathfrak g}_{1/2}+\dim{\mathfrak g}_1\le 16.
$$
\end{remark}
\vspace{0.1cm}

{\bf Case (7).} Suppose that $k=5$, $n=5$. In this situation inequality (\ref{estim2}) implies $\dim {\mathfrak g}(\Omega)\ge 10$, which by Lemma \ref{ourlemma} yields that $\Omega$ is linearly equivalent to the circular cone $C_5$. Hence, after a linear change of variables, $S(\Omega,H)$ turns into the domain $T_5$ defined in (\ref{domaint5}). However, $d(T_5)=21>20=n^2-5$, so this case makes no contributions to the classification of homogeneous hyperbolic $n$-dimensional manifolds with automorphism group dimension $n^2-5$.

The proof of Theorem \ref{main1} is complete.

\section{Proof of Theorem \ref{main2}}\label{proof2}
\setcounter{equation}{0}

As before, we utilize the fact that, by \cite{VGP-S}, \cite{N2}, the manifold $M$ is biholomorphic to an affinely homogeneous Siegel domain of the second kind $S(\Omega,H)$. Since one has $n^2-6\ge 2n$, it follows that $n\ge 4$. Also, as $M$ is not biholomorphic to $B^n$, we have $k\ge 2$. Now, it is not hard to see that Lemma \ref{n5k31} holds in this situation as well, which again leads to the seven cases stated in Section \ref{proof1}. 

{\bf Case (1).} The values $m=0,1,2,3$ are treated as before and yield no domains. However, this time estimate (\ref{estim2}) implies $s\ge n^2-4n-4$, which also allows for $m=4$. This possibility leads to two additional subcases: (e) where $n=6$ with $\lambda_1=\lambda_2\ne\lambda_3=\lambda_4$, and (f) where $n=7$ with $\lambda_1\ne\lambda_2=\lambda_3=\lambda_4=\lambda_5$. Case (1) is then easily seen to contribute the products $B^3\times B^3$ and $B^2\times B^5$ to the classification, which arise from subcases (e) and (f), respectively, with $d(B^3\times B^3)=30=n^2-6$, $d(B^2\times B^5)=43=n^2-6$.
\vspace{0.1cm}  

{\bf Case (2).} It is not hard to observe that this case only contributes to our classification the domain $D_4$ with $v=(1,1,0)$ (see (\ref{domaind8})), which is exactly the domain ${\mathcal D}$ defined in (\ref{psmathcald}). As we have already mentioned, ${\mathcal D}$ is linearly equivalent to the well-known example of a bounded non-symmetric homogeneous domain in $\CC^4$ given by I.~Pyatetskii-Shapiro (see \cite[pp.~26--28]{P-S}). Here $d({\mathcal D})=10=n^2-6$.
\vspace{0.1cm}

{\bf Case (3).} Here the domain $D_5$ defined in (\ref{domaind9}) leads to the product of unit balls $B^1\times B^2\times B^2$ with $d(B^1\times B^2\times B^2)=19=n^2-6$. 

Further, it is clear from the analysis given in Section \ref{proof1} that for the domain $D_6$ defined in (\ref{domaind10}) we only need to study the situation when $s=1$. We will show:

\begin{lemma}\label{oneofthemissmall} \it
For the domain $D_6$ with $s=1$ and $\mathfrak{g}=\mathfrak{g}(D_6)$ one has $\dim\mathfrak{g}_{1/2}\le 2$.
\end{lemma}

\begin{proof} Let us write the $\Omega_3$-Hermitian form ${\mathcal H}$ as
$$
{\mathcal H}=u|w_1|^2+v|w_2|^2+a\bar w_1w_2+\bar a\, \bar w_2 w_1,
$$
where $u,v\in\RR^3$ and $a\in\CC^3$. It is then clear that $u,v\in\bar\Omega_3\setminus\{0\}$. We will consider two cases.

{\bf Case (i).} Suppose first that $u\in\Omega_3$. Then, as the cone $\Omega_3$ is homogeneous, we may assume that $u=(1,0,0)$. Further, replacing $w_1$ by $w_1+a_1w_2$, we may suppose that $a_1=0$. In addition, rotating the variables $z_2, z_3$ by a transformation from $\O_2$, we can always reduce to the case when ${\mathcal H}_3$ has no $|w_2|^2$-term, i.e., when $v_3=0$. As $s=1$, we must have $a_3\ne 0$, hence, by scaling $w_2$, one can also assume that $a_3=1$. 

To utilize Theorem \ref{descrg1/2}, let $\Phi:\CC^3\to\CC^2$ be a $\CC$-linear map
\begin{equation}
\Phi(z_1,z_2,z_3)=\left(\varphi_1^1z_1+\varphi_2^1z_2+\varphi_3^1z_3, \varphi_1^2z_1+\varphi_2^2z_2+\varphi_3^2z_3\right),\label{mapPhineww}
\end{equation}
where $\varphi_i^j\in\CC$. Fixing ${\mathbf w}\in\CC^2$, for $x\in\RR^3$ we compute
$$
\begin{array}{l}
\displaystyle{\mathcal H}({\mathbf w},\Phi(x))=\left(\bar{\mathbf w}_1(\varphi_1^1x_1+\varphi_2^1x_2+\varphi_3^1x_3)+
v_1\bar{\mathbf w}_2(\varphi_1^2x_1+\varphi_2^2x_2+\varphi_3^2x_3),\right.\\
\vspace{-0.1cm}\\
\hspace{2.3cm}\left.v_2\bar{\mathbf w}_2(\varphi_1^2x_1+\varphi_2^2x_2+\varphi_3^2x_3)+a_2\bar{\mathbf w}_1(\varphi_1^2x_1+\varphi_2^2x_2+\varphi_3^2x_3)+\right.\\
\vspace{-0.1cm}\\
\hspace{2.3cm}\left.\bar a_2\bar{\mathbf w}_2(\varphi_1^1x_1+\varphi_2^1x_2+\varphi_3^1x_3),\bar{\mathbf w}_1(\varphi_1^2x_1+\varphi_2^2x_2+\varphi_3^2x_3)+\right.\\
\vspace{-0.1cm}\\
\hspace{2.3cm}\left.\bar{\mathbf w}_2(\varphi_1^1x_1+\varphi_2^1x_2+\varphi_3^1x_3)\right)=\left((\varphi_1^1\bar{\mathbf w}_1+v_1\varphi_1^2\bar{\mathbf w}_2)x_1+\right.\\
\vspace{-0.1cm}\\
\hspace{2.3cm}\left.(\varphi_2^1\bar{\mathbf w}_1+v_1\varphi_2^2\bar{\mathbf w}_2)x_2+(\varphi_3^1\bar{\mathbf w}_1+v_1\varphi_3^2\bar{\mathbf w}_2)x_3, (a_2\varphi_1^2\bar{\mathbf w}_1+\right.\\
\vspace{-0.1cm}\\
\hspace{2.3cm}\left.(\bar a_2\varphi_1^1+v_2\varphi_1^2)\bar{\mathbf w}_2)x_1+(a_2\varphi_2^2\bar{\mathbf w}_1+(\bar a_2\varphi_2^1+v_2\varphi_2^2)\bar{\mathbf w}_2)x_2+\right.\\
\vspace{-0.1cm}\\
\hspace{2.3cm}\left.(a_2\varphi_3^2\bar{\mathbf w}_1+(\bar a_2\varphi_3^1+v_2\varphi_3^2)\bar{\mathbf w}_2)x_3, (\varphi_1^2\bar{\mathbf w}_1+\varphi_1^1\bar{\mathbf w}_2)x_1+\right.\\
\vspace{-0.1cm}\\
\hspace{2.3cm}\left.(\varphi_2^2\bar{\mathbf w}_1+\varphi_2^1\bar{\mathbf w}_2)x_2+(\varphi_3^2\bar{\mathbf w}_1+\varphi_3^1\bar{\mathbf w}_2)x_3\right).
\end{array}
$$
Then from formula (\ref{Phiw0}) we see
$$
\begin{array}{l}
\displaystyle\Phi_{{\mathbf w}}(x)=\left((\Im(\varphi_1^1\bar{\mathbf w}_1)+v_1\Im(\varphi_1^2\bar{\mathbf w}_2))x_1+(\Im(\varphi_2^1\bar{\mathbf w}_1)+v_1\Im(\varphi_2^2\bar{\mathbf w}_2))x_2+\right.\\
\vspace{-0.1cm}\\
\hspace{0.5cm}\left.(\Im(\varphi_3^1\bar{\mathbf w}_1)+v_1\Im(\varphi_3^2\bar{\mathbf w}_2))x_3, (\Im(a_2\varphi_1^2\bar{\mathbf w}_1)+\Im((\bar a_2\varphi_1^1+v_2\varphi_1^2)\bar{\mathbf w}_2))x_1+\right.\\
\vspace{-0.1cm}\\
\hspace{0.5cm}\left.(\Im(a_2\varphi_2^2\bar{\mathbf w}_1)+\Im((\bar a_2\varphi_2^1+v_2\varphi_2^2)\bar{\mathbf w}_2))x_2+(\Im(a_2\varphi_3^2\bar{\mathbf w}_1)+\Im((\bar a_2\varphi_3^1+\right.\\
\vspace{-0.1cm}\\
\hspace{0.5cm}\left.v_2\varphi_3^2)\bar{\mathbf w}_2))x_3, (\Im(\varphi_1^2\bar{\mathbf w}_1)+\Im(\varphi_1^1\bar{\mathbf w}_2))x_1+(\Im(\varphi_2^2\bar{\mathbf w}_1)+\Im(\varphi_2^1\bar{\mathbf w}_2))x_2+\right.\\
\vspace{-0.1cm}\\
\hspace{7.2cm}\left.(\Im(\varphi_3^2\bar{\mathbf w}_1)+\Im(\varphi_3^1\bar{\mathbf w}_2))x_3\right).
\end{array}
$$

Using (\ref{alggomega3}), we then see that the condition that $\Phi_{{\mathbf w}}$ lies in ${\mathfrak g}(\Omega_3)$ for every ${\mathbf w}\in\CC^2$ leads to the relations
\begin{equation}
\begin{array}{l}
\varphi_1^1=a_2\varphi_2^2=\varphi_3^2,\,\,v_1\varphi_1^2=\bar a_2\varphi_2^1+v_2\varphi_2^2=\varphi_3^1,\\
\vspace{-0.3cm}\\
\varphi_2^1=a_2\varphi_1^2,\,\,v_1\varphi_2^2=\bar a_2\varphi_1^1+v_2\varphi_1^2,\,\,\varphi_3^1=\varphi_1^2,\\
\vspace{-0.3cm}\\
\varphi_1^1=v_1\varphi_3^2,\,\,\varphi_2^2=-a_2\varphi_3^2,\,\,\varphi_2^1=-\bar a_2\varphi_3^1-v_2\varphi_3^2.
\end{array}\label{relmapphibig}
\end{equation}
If $a_2=0$, it immediately follows that $\Phi=0$, thus by formula (\ref{cond1}) we have $\mathfrak{g}_{1/2}=0$. Suppose now that $a_2\ne 0$. It is then straightforward to see that, unless $v_1=1$, $v_2=0$, $a_2=\pm i$, the space of all maps $\Phi$ satisfying relations (\ref{relmapphibig}) has complex dimension at most 1, which by (\ref{cond1}) implies $\dim\mathfrak{g}_{1/2}\le 2$. 

We thus assume that
$$
{\mathcal H}=(|w_1|^2+|w_2|^2,\pm i(\bar w_1w_2-\bar w_2 w_1), \bar w_1w_2+\bar w_2 w_1).
$$
Changing the $w$-variables as
$$
w_1\mapsto -\frac{i}{\sqrt{2}}(w_1+iw_2),\,\,w_2\mapsto \frac{1}{\sqrt{2}}(w_1-iw_2),
$$
we can suppose that
$$
{\mathcal H}=(|w_1|^2+|w_2|^2,\mp(|w_1|^2-|w_2|^2), \bar w_1w_2+\bar w_2 w_1).
$$
Further, swapping $w_1$ and $w_2$ if necessary, we reduce our considerations to the case where
\begin{equation}
{\mathcal H}=(|w_1|^2+|w_2|^2,|w_1|^2-|w_2|^2, \bar w_1w_2+\bar w_2 w_1).\label{verynewformcalh}
\end{equation}

We will now show that for the above $\Omega_3$-Hermitian form ${\mathcal H}$ one has $\mathfrak{g}_{1/2}=0$. Consider a map $\Phi:\CC^3\to\CC^2$ as in (\ref{mapPhineww}), fix ${\mathbf w}\in\CC^2$, and for $x\in\RR^3$ compute
$$
\begin{array}{l}
\displaystyle{\mathcal H}({\mathbf w},\Phi(x))=\left(\bar{\mathbf w}_1(\varphi_1^1x_1+\varphi_2^1x_2+\varphi_3^1x_3)+
\bar{\mathbf w}_2(\varphi_1^2x_1+\varphi_2^2x_2+\varphi_3^2x_3),\right.\\
\vspace{-0.1cm}\\
\hspace{2.3cm}\left.\bar{\mathbf w}_1(\varphi_1^1x_1+\varphi_2^1x_2+\varphi_3^1x_3)-\bar{\mathbf w}_2(\varphi_1^2x_1+\varphi_2^2x_2+\varphi_3^2x_3),\right.\\
\vspace{-0.1cm}\\
\hspace{2.3cm}\left.\bar{\mathbf w}_1(\varphi_1^2x_1+\varphi_2^2x_2+\varphi_3^2x_3)+\bar{\mathbf w}_2(\varphi_1^1x_1+\varphi_2^1x_2+\varphi_3^1x_3)\right)=\\
\vspace{-0.1cm}\\
\hspace{2.3cm}\left((\varphi_1^1\bar{\mathbf w}_1+\varphi_1^2\bar{\mathbf w}_2)x_1+(\varphi_2^1\bar{\mathbf w}_1+\varphi_2^2\bar{\mathbf w}_2)x_2+(\varphi_3^1\bar{\mathbf w}_1+\varphi_3^2\bar{\mathbf w}_2)x_3,\right.\\
\vspace{-0.1cm}\\
\hspace{2.3cm}\left.(\varphi_1^1\bar{\mathbf w}_1-\varphi_1^2\bar{\mathbf w}_2)x_1+(\varphi_2^1\bar{\mathbf w}_1-\varphi_2^2\bar{\mathbf w}_2)x_2+(\varphi_3^1\bar{\mathbf w}_1-\varphi_3^2\bar{\mathbf w}_2)x_3,\right.\\
\vspace{-0.1cm}\\
\hspace{2.3cm}\left.(\varphi_1^2\bar{\mathbf w}_1+\varphi_1^1\bar{\mathbf w}_2)x_1+(\varphi_2^2\bar{\mathbf w}_1+\varphi_2^1\bar{\mathbf w}_2)x_2+(\varphi_3^2\bar{\mathbf w}_1+\varphi_3^1\bar{\mathbf w}_2)x_3\right).
\end{array}
$$
Then from formula (\ref{Phiw0}) we see
$$
\begin{array}{l}
\displaystyle\Phi_{{\mathbf w}}(x)=\left((\Im(\varphi_1^1\bar{\mathbf w}_1)+\Im(\varphi_1^2\bar{\mathbf w}_2))x_1+(\Im(\varphi_2^1\bar{\mathbf w}_1)+\Im(\varphi_2^2\bar{\mathbf w}_2))x_2+\right.\\
\vspace{-0.1cm}\\
\hspace{1.4cm}\left.(\Im(\varphi_3^1\bar{\mathbf w}_1)+\Im(\varphi_3^2\bar{\mathbf w}_2))x_3,(\Im(\varphi_1^1\bar{\mathbf w}_1)-\Im(\varphi_1^2\bar{\mathbf w}_2))x_1+\right.\\
\vspace{-0.1cm}\\
\hspace{1.4cm}\left.(\Im(\varphi_2^1\bar{\mathbf w}_1)-\Im(\varphi_2^2\bar{\mathbf w}_2))x_2+(\Im(\varphi_3^1\bar{\mathbf w}_1)-\Im(\varphi_3^2\bar{\mathbf w}_2))x_3,\right.\\
\vspace{-0.1cm}\\
\hspace{1.4cm}\left.(\Im(\varphi_1^2\bar{\mathbf w}_1)+\Im(\varphi_1^1\bar{\mathbf w}_2))x_1+(\Im(\varphi_2^2\bar{\mathbf w}_1)+\Im(\varphi_2^1\bar{\mathbf w}_2))x_2+\right.\\
\vspace{-0.1cm}\\
\hspace{6.5cm}\left.(\Im(\varphi_3^2\bar{\mathbf w}_1)+\Im(\varphi_3^1\bar{\mathbf w}_2))x_3\right).
\end{array}
$$
From (\ref{alggomega3}) we then see that the condition that $\Phi_{{\mathbf w}}$ lies in ${\mathfrak g}(\Omega_3)$ for every ${\mathbf w}\in\CC^2$ leads to the relations
\begin{equation}
\begin{array}{l}
\varphi_1^1=\varphi_2^1=\varphi_3^2,\,\,\varphi_1^2=-\varphi_2^2=\varphi_3^1,\\
\end{array}\label{relmapphibig11}
\end{equation}

Further, let $c$ be a symmetric $\CC$-bilinear form on $\CC^2$ with values in $\CC^2$:
$$
c(w,w)=\left(c^1_{11}w_1^2+2c^1_{12}w_1w_2+c^1_{22}w_2^2,c^2_{11}w_1^2+2c^2_{12}w_1w_2+c^2_{22}w_2^2\right),
$$
where $c^{\ell}_{ij}\in\CC$. Then for $w,w'\in\CC^2$ using (\ref{verynewformcalh}) we calculate
\begin{equation}
\makebox[250pt]{$\begin{array}{l}
{\mathcal H}(w,c(w',w'))=\left(\bar w_1(c^1_{11}(w_1')^{2}+2c^1_{12}w_1'w_2'+c^1_{22}(w_2')^{2})+\bar w_2(c^2_{11}(w_1')^{2}+\right.\\
\vspace{-0.1cm}\\
\hspace{0.8cm}\left.2c^2_{12}w_1'w_2'+c^2_{22}(w_2')^{2}), \bar w_1(c^1_{11}(w_1')^{2}+2c^1_{12}w_1'w_2'+c^1_{22}(w_2')^{2})-\right.\\
\vspace{-0.1cm}\\
\hspace{0.8cm}\left.\bar w_2(c^2_{11}(w_1')^{2}+2c^2_{12}w_1'w_2'+c^2_{22}(w_2')^{2}), \bar w_1(c^2_{11}(w_1')^{2}+2c^2_{12}w_1'w_2'+\right.\\
\vspace{-0.1cm}\\
\hspace{0.8cm}\left.c^2_{22}(w_2')^{2})+\bar w_2(c^1_{11}(w_1')^{2}+2c^1_{12}w_1'w_2'+c^1_{22}(w_2')^{2})\right).
\end{array}$}\label{exprssdiff1}
\end{equation}
On the other hand, we have
$$
\begin{array}{l}
\Phi({\mathcal H}(w',w))=\left(\varphi_1^1(\bar w_1'w_1+\bar w_2'w_2)+\varphi_2^1(\bar w_1'w_1-\bar w_2'w_2)+\varphi_3^1(\bar w_1' w_2+\bar w_2'w_1),\right.\\
\vspace{-0.1cm}\\
\hspace{2.4cm}\left.\varphi_1^2(\bar w_1'w_1+\bar w_2'w_2)+\varphi_2^2(\bar w_1'w_1-\bar w_2'w_2)+\varphi_3^2(\bar w_1' w_2+\bar w_2'w_1)\right)=\\
\vspace{-0.1cm}\\
\hspace{2.4cm}\left((\varphi_1^1+\varphi_2^1)\bar w_1'w_1+(\varphi_1^1-\varphi_2^1)\bar w_2'w_2+\varphi_3^1(\bar w_1' w_2+\bar w_2'w_1),\right.\\
\vspace{-0.1cm}\\
\hspace{2.4cm}\left.(\varphi_1^2+\varphi_2^2)\bar w_1'w_1+(\varphi_1^2-\varphi_2^2)\bar w_2'w_2+\varphi_3^2(\bar w_1' w_2+\bar w_2'w_1)\right). 
\end{array}
$$
Therefore
\begin{equation}
\makebox[250pt]{$\begin{array}{l}
2i{\mathcal H}(\Phi({\mathcal H}(w',w)),w')=\\
\vspace{-0.1cm}\\
\hspace{2cm}
2i\left(w_1'\left((\bar\varphi_1^1+\bar\varphi_2^1)w_1'\bar w_1+(\bar\varphi_1^1-\bar\varphi_2^1) w_2'\bar w_2+\bar\varphi_3^1(w_1' \bar w_2+w_2'\bar w_1)\right)+\right.\\
\vspace{-0.1cm}\\
\hspace{2cm}\left.w_2'\left((\bar\varphi_1^2+\bar\varphi_2^2)w_1'\bar w_1+(\bar\varphi_1^2-\bar\varphi_2^2) w_2'\bar w_2+\bar\varphi_3^2(w_1' \bar w_2+w_2'\bar w_1)\right),\right.\\
\vspace{-0.1cm}\\
\hspace{2cm}\left.w_1'\left((\bar\varphi_1^1+\bar\varphi_2^1)w_1'\bar w_1+(\bar\varphi_1^1-\bar\varphi_2^1) w_2'\bar w_2+\bar\varphi_3^1(w_1' \bar w_2+w_2'\bar w_1)\right)-\right.\\
\vspace{-0.1cm}\\
\hspace{2cm}\left.w_2'\left((\bar\varphi_1^2+\bar\varphi_2^2)w_1'\bar w_1+(\bar\varphi_1^2-\bar\varphi_2^2) w_2'\bar w_2+\bar\varphi_3^2(w_1' \bar w_2+w_2'\bar w_1)\right),\right.\\
\vspace{-0.1cm}\\
\hspace{2cm}\left. w_1'\left((\bar\varphi_1^2+\bar\varphi_2^2)w_1'\bar w_1+(\bar\varphi_1^2-\bar\varphi_2^2) w_2'\bar w_2+\bar\varphi_3^2(w_1' \bar w_2+w_2'\bar w_1)\right)+\right.\\
\vspace{-0.1cm}\\
\hspace{2cm}\left.w_2'\left((\bar\varphi_1^1+\bar\varphi_2^1)w_1'\bar w_1+(\bar\varphi_1^1-\bar\varphi_2^1) w_2'\bar w_2+\bar\varphi_3^1(w_1' \bar w_2+w_2'\bar w_1)\right)\right).
\end{array}$}\label{expressdiff2}
\end{equation}
Let us now compare expressions (\ref{exprssdiff1}) and (\ref{expressdiff2}) as required by condition (\ref{cond1}). Specifically, looking at the coefficients at $(w_2')^2\bar w_1$ and $(w_1')^2\bar w_2$ in the first and second components of these expressions, we obtain the identities:
$$
c_{22}^1=2i\bar\varphi_3^2,\qquad c_{22}^1=-2i\bar\varphi_3^2,\qquad c_{11}^2=2i\bar\varphi_3^1,\qquad -c_{11}^2=2i\bar\varphi_3^1,
$$
which imply $\varphi_3^1=0$, $\varphi_3^2=0$. Taken together with (\ref{relmapphibig11}), these conditions yield $\Phi=0$, hence ${\mathfrak g}_{1/2}=0$ as required.

{\bf Case (ii).} Suppose now that $u\in\partial\Omega_3\setminus\{0\}$. In this situation, as the group $G(\Omega_3)^{\circ}=\RR_{+}\times\SO_{1,2}^{\circ}$ acts transitively on $\partial\Omega_3\setminus\{0\}$, we may assume that\linebreak $u=(1,1,0)$. Further, replacing $w_1$ by $w_1+a_1w_2$, we may suppose that $a_1=0$. 

Let $\Phi:\CC^3\to\CC^2$ be a $\CC$-linear map as in (\ref{mapPhineww}). Fixing ${\mathbf w}\in\CC^2$, for $x\in\RR^3$ we compute
$$
\begin{array}{l}
\displaystyle{\mathcal H}({\mathbf w},\Phi(x))=\left(\bar{\mathbf w}_1(\varphi_1^1x_1+\varphi_2^1x_2+\varphi_3^1x_3)+
v_1\bar{\mathbf w}_2(\varphi_1^2x_1+\varphi_2^2x_2+\varphi_3^2x_3),\right.\\
\vspace{-0.1cm}\\
\left.\hspace{1cm}\bar{\mathbf w}_1(\varphi_1^1x_1+\varphi_2^1x_2+\varphi_3^1x_3)+v_2\bar{\mathbf w}_2(\varphi_1^2x_1+\varphi_2^2x_2+\varphi_3^2x_3)+\right.\\
\vspace{-0.1cm}\\
\left.\hspace{1cm}a_2\bar{\mathbf w}_1(\varphi_1^2x_1+\varphi_2^2x_2+\varphi_3^2x_3)+\bar a_2\bar{\mathbf w}_2(\varphi_1^1x_1+\varphi_2^1x_2+\varphi_3^1x_3),\right.\\
\vspace{-0.1cm}\\
\left.\hspace{1cm}
v_3\bar{\mathbf w}_2(\varphi_1^2x_1+\varphi_2^2x_2+\varphi_3^2x_3)+a_3\bar{\mathbf w}_1(\varphi_1^2x_1+\varphi_2^2x_2+\varphi_3^2x_3)+\right.\\
\vspace{-0.1cm}\\
\left.\hspace{1cm}\bar a_3\bar{\mathbf w}_2(\varphi_1^1x_1+\varphi_2^1x_2+\varphi_3^1x_3)\right)=\left((\varphi_1^1\bar{\mathbf w}_1+v_1\varphi_1^2\bar{\mathbf w}_2)x_1+\right.\\
\vspace{-0.1cm}\\
\hspace{1cm}\left.(\varphi_2^1\bar{\mathbf w}_1+v_1\varphi_2^2\bar{\mathbf w}_2)x_2+(\varphi_3^1\bar{\mathbf w}_1+v_1\varphi_3^2\bar{\mathbf w}_2)x_3, ( (\varphi_1^1+a_2\varphi_1^2)\bar{\mathbf w}_1+\right.\\
\vspace{-0.1cm}\\
\hspace{1cm}\left.(\bar a_2\varphi_1^1+v_2\varphi_1^2)\bar{\mathbf w}_2)x_1+((\varphi_2^1+a_2\varphi_2^2)\bar{\mathbf w}_1+(\bar a_2\varphi_2^1+v_2\varphi_2^2)\bar{\mathbf w}_2)x_2+\right.\\
\vspace{-0.1cm}\\
\hspace{1cm}\left.((\varphi_3^1+a_2\varphi_3^2)\bar{\mathbf w}_1+(\bar a_2\varphi_3^1+v_2\varphi_3^2)\bar{\mathbf w}_2)x_3, (a_3\varphi_1^2\bar{\mathbf w}_1+(\bar a_3\varphi_1^1+v_3\varphi_1^2)\bar{\mathbf w}_2)x_1+\right.\\
\vspace{-0.1cm}\\
\hspace{1cm}\left.(a_3\varphi_2^2\bar{\mathbf w}_1+(\bar a_3\varphi_2^1+v_3\varphi_2^2)\bar{\mathbf w}_2)x_2+(a_3\varphi_3^2\bar{\mathbf w}_1+(\bar a_3\varphi_3^1+v_3\varphi_3^2)\bar{\mathbf w}_2)x_3\right).
\end{array}
$$
Then from formula (\ref{Phiw0}) we see
$$
\hspace{-0.3cm}\begin{array}{l}
\displaystyle\Phi_{{\mathbf w}}(x)=\left((\Im(\varphi_1^1\bar{\mathbf w}_1)+v_1\Im(\varphi_1^2\bar{\mathbf w}_2))x_1+(\Im(\varphi_2^1\bar{\mathbf w}_1)+v_1\Im(\varphi_2^2\bar{\mathbf w}_2))x_2+\right.\\
\vspace{-0.1cm}\\
\hspace{0.1cm}\left.
(\Im(\varphi_3^1\bar{\mathbf w}_1)+v_1\Im(\varphi_3^2\bar{\mathbf w}_2))x_3, ( \Im((\varphi_1^1+a_2\varphi_1^2)\bar{\mathbf w}_1)+\Im((\bar a_2\varphi_1^1+v_2\varphi_1^2)\bar{\mathbf w}_2))x_1+\right.\\
\vspace{-0.1cm}\\
\hspace{0.1cm}\left.
(\Im((\varphi_2^1+a_2\varphi_2^2)\bar{\mathbf w}_1)+\Im((\bar a_2\varphi_2^1+v_2\varphi_2^2)\bar{\mathbf w}_2))x_2+(\Im((\varphi_3^1+a_2\varphi_3^2)\bar{\mathbf w}_1)+\right.\\
\vspace{-0.1cm}\\
\hspace{0.1cm}\left.\Im((\bar a_2\varphi_3^1+v_2\varphi_3^2)\bar{\mathbf w}_2))x_3, (\Im(a_3\varphi_1^2\bar{\mathbf w}_1)+\Im((\bar a_3\varphi_1^1+v_3\varphi_1^2)\bar{\mathbf w}_2))x_1+\right.\\
\vspace{-0.1cm}\\
\hspace{0.1cm}\left.(\Im(a_3\varphi_2^2\bar{\mathbf w}_1)+\Im((\bar a_3\varphi_2^1+v_3\varphi_2^2)\bar{\mathbf w}_2))x_2+(\Im(a_3\varphi_3^2\bar{\mathbf w}_1)+\Im((\bar a_3\varphi_3^1+v_3\varphi_3^2)\bar{\mathbf w}_2))x_3\right).
\end{array}
$$

Using (\ref{alggomega3}), we then see that the condition that $\Phi_{{\mathbf w}}$ lies in ${\mathfrak g}(\Omega_3)$ for every ${\mathbf w}\in\CC^2$ leads to the relations
\begin{equation}
\begin{array}{l}
\varphi_1^1=\varphi_2^1+a_2\varphi_2^2=a_3\varphi_3^2,\,\,v_1\varphi_1^2=\bar a_2\varphi_2^1+v_2\varphi_2^2=\bar a_3\varphi_3^1+v_3\varphi_3^2,\\
\vspace{-0.1cm}\\
\varphi_2^1=\varphi_1^1+a_2\varphi_1^2,\,\,v_1\varphi_2^2=\bar a_2\varphi_1^1+v_2\varphi_1^2,\,\,\varphi_3^1=a_3\varphi_1^2,\\
\vspace{-0.1cm}\\
\bar a_3\varphi_1^1+v_3\varphi_1^2=v_1\varphi_3^2,\,\,a_3\varphi_2^2=-\varphi_3^1-a_2\varphi_3^2,\,\,\bar a_3\varphi_2^1+v_3\varphi_2^2=-\bar a_2\varphi_3^1-v_2\varphi_3^2.
\end{array}\label{relmapphibig1}
\end{equation}
It easily follows from (\ref{relmapphibig1}) that if $a_3=0$, then $\Phi=0$, so by formula (\ref{cond1}) we have ${\mathfrak g}_{1/2}=0$. If $a_3\ne 0$, then, by scaling $w_2$, we can assume that $a_3=1$. In this situation it is straightforward to see that, unless $v_1=1$, $v_2=-1$, $v_3=0$, $a_2=0$, the space of all maps $\Phi$ satisfying relations (\ref{relmapphibig1}) has complex dimension at most 1; by formula (\ref{cond1}) this implies $\dim\mathfrak{g}_{1/2}\le 2$. Notice now that for the above values of $v_1$, $v_2$, $v_3$, $a_2$ the form ${\mathcal H}$ coincides with the right-hand side of (\ref{verynewformcalh}), for which we have already shown that ${\mathfrak g}_{1/2}=0$. 

The proof of the lemma is now complete.\end{proof}

Now, Lemma \ref{oneofthemissmall} together with (\ref{estim 8}) and the second inequality in (\ref{estimm}) yields
$d(D_6)\le  17<19=n^2-6$. Thus, we have shown that Case (3) only contributes the product $B^1\times B^2\times B^2$ to our classification.
\vspace{0.1cm}

{\bf Case (4).} Here inequality (\ref{estim2}) implies $s+\dim{\mathfrak g}(\Omega)\ge 12$, so we need to look at the possibility when $s=9$ and $\Omega$ is linearly equivalent to $\Omega_2$. This possibility yields the product $B^1\times B^1\times B^4$ with $d(B^1\times B^1\times B^4)=30=n^2-6$. Also, it follows from the analysis given in Section \ref{proof1} that if $\Omega$ is linearly equivalent to $\Omega_3$, Case (4) makes no contributions to the classification. Indeed, either the domain $D_7$ defined in (\ref{domaind7new}) is not homogeneous or we have $d(D_7)=22<30=n^2-6$.
\vspace{0.1cm}

{\bf Case (5).} Clearly, this case makes no contributions to our classification.
\vspace{0.1cm}

{\bf Case (6).} Here inequality (\ref{estim2}) yields $\dim {\mathfrak g}(\Omega)\ge 6$, hence it follows, for example, by Lemma \ref{ourlemma}, that the cone $\Omega$ is linearly equivalent to the circular cone $C_4=\Omega_6$. Then, arguing as in Section \ref{proof1}, we see that Case (6) does not contribute any domains to the classification. Indeed, either the domain $D_8$ defined in (\ref{domaind12}) is not homogeneous or we have $d(D_8)\le 16<19=n^2-6$.
\vspace{0.1cm}

{\bf Case (7).} In this case, inequality (\ref{estim2}) implies $\dim {\mathfrak g}(\Omega)\ge 9$, so, as in Section \ref{proof1}, Lemma \ref{ourlemma} yields that $\Omega$ is linearly equivalent to the circular cone $C_5$. Since $d(T_5)=21>19=n^2-6$, this again leads to the conclusion that this case makes no contributions to the classification of homogeneous hyperbolic $n$-dimensional manifolds with automorphism group dimension $n^2-6$.

The proof of Theorem \ref{main2} is now complete.

\end{document}